\documentclass[12pt,a4paper]{article}
\usepackage[ngerman, english]{babel}
\usepackage{enumerate}
\usepackage{amsmath,amssymb, amsthm, amsfonts}
\usepackage[colorlinks]{hyperref}
\usepackage{mathabx}

\theoremstyle{plain}
\newtheorem {lemma}{Lemma}
\newtheorem {proposition}[lemma]{Proposition}
\newtheorem {theorem}[lemma]{Theorem}

\newtheorem {corollary}[lemma]{Corollary}

\theoremstyle{definition}
\newtheorem{definition}[lemma]{Definition}
\newtheorem{remark}[lemma]{Remark}

\newcommand{\N}{\mathbb{N}}

\DeclareMathOperator*{\plus}{\overset{.}{\bigplus}}
\newcommand{\+}{\overset{.}{+}}
\newcommand{\minus}{\overset{.}{-}}
\newcommand{\h}{\mathfrak{H}}
\newcommand{\tr}{\operatorname{tr}}
\newcommand{\hb}{\operatorname{hb}}
\newcommand{\GL}{\operatorname{GL}}

\newcommand{\E}{\operatorname{E}}
\newcommand{\C}{\operatorname{C}}
\newcommand{\U}{\operatorname{U}}
\newcommand{\EU}{\operatorname{EU}}
\newcommand{\NU}{\operatorname{NU}}
\newcommand{\CU}{\operatorname{CU}}
\newcommand{\Ort}{\operatorname{O}}
\newcommand{\Sp}{\operatorname{Sp}}
\newcommand{\Normaliser}{\operatorname{Normaliser}}

\newcommand{\FP}{\operatorname{FP}}

\title{The subnormal structure of classical-like groups over commutative rings}
\author{Raimund Preusser}
\date{}
\AtEndDocument{\bigskip{\footnotesize%
\textsc{Chebyshev Laboratory, St. Petersburg State University, Russia}\\ 
\textit{E-mail address:} \texttt{raimund.preusser@gmx.de}\\
}}

\begin{document}
\maketitle
\begin{abstract}
Let $n$ be an integer greater than or equal to $3$ and $(R,\Delta)$ a Hermitian form ring where $R$ is commutative. We prove that if $H$ is a subgroup of the odd-dimensional unitary group $\U_{2n+1}(R,\Delta)$ normalised by a relative elementary subgroup $\EU_{2n+1}((R,\Delta),(I,\Omega))$, then there is an odd form ideal $(J,\Sigma)$ such that 
$\EU_{2n+1}((R,\Delta),(JI^{k},\Omega_{\min}^{JI^k}\+\Sigma\circ I^{k}))\leq H \leq \CU_{2n+1}((R,\Delta),(J,\Sigma))$
where $k=12$ if $n=3$ respectively $k=10$ if $n\geq 4$. As a conseqence of this result we obtain a sandwich theorem for subnormal subgroups of odd-dimensional unitary groups.
\end{abstract}
\let\thefootnote\relax
\let\thefootnote\relax
\let\thefootnote\relax\footnotetext{The work is supported by the Russian Science Foundation grant 19-71-30002.}

\section{Introduction}
Recall that if $H$ is a subgroup of a group $G$ and $d$ is a nonnegative integer, then one writes $H\lhd_d G$ if $H=H_0\lhd H_1\lhd \dots\lhd H_{d-1}\lhd H_d=G$ for some subgroups $H_1,\dots,H_{d-1}$ of $G$. If $H\lhd_d G$ for some $d$, then $H$ is called a \textit{subnormal} subgroup of $G$. In 1972, J. Wilson \cite{20} proved that if $H\lhd_d \GL_n(R)$ where $n\geq 3$ and $R$ is a commutative ring, then there is an ideal $I$ such that 
\begin{equation}
\E_n(R,I^k)\leq H\leq \C_n(R,I)
\end{equation}
where $k=\frac{88(28^{d-1})-7}{27}$ if $n=3$ resp. $k=\frac{7^d-1}{6}$ if $n\geq 4$. In (1), $\E_n(R,I^k)$ denotes the relative elementary subgroup of level $I^k$ and $\C_n(R,I)$ the full congruence subgroup of level $I$. Wilson's result  was subsequently improved by L. Vaserstein \cite{vaserstein_sub,vaserstein_sub2} and N. Vavilov \cite{vavilov_sub}. For more information on the general linear case we refer the reader to the introduction in \cite{vavilov_sub2}.

A natural question is if similar results hold for other classical and classical-like groups. The (even-dimensional) hyperbolic unitary groups $\U_{2n}(R,\Lambda)$ were defined by A. Bak in 1969 \cite{bak_thesis}. They embrace the classical Chevalley groups of type $C_m$ and $D_m$, namely the even-dimensional symplectic and orthogonal groups $\Sp_{2n}(R)$ and $\Ort_{2n}(R)$. In 2006, Z. Zhang \cite{zhang_sub1} obtained a ``sandwich'' result similar to (1) for subnormal subgroups of stable hyperbolic unitary groups  $\U(R,\Lambda)=\varinjlim_{n\geq 1} \U_{2n}(R,\Lambda)$. In 2012, H. You \cite{you} obtained a sandwich result for subnormal subgroups of nonstable hyperbolic unitary groups.

In 2018, A. Bak and the author \cite{bak-preusser} defined odd-dimensional unitary groups $\U_{2n+1}(R,$ $\Delta)$. These groups generalise the even-dimensional unitary groups $\U_{2n}(R,\Lambda)$ and embrace \emph{all} classical Chevalley groups. In this article we prove that if $H\lhd_d \U_{2n+1}(R,\Delta)$ where $(R,\Delta)$ is a commutative Hermitian form ring and $n\geq 3$, then there is an odd form ideal $(I,\Omega)$ such that
\begin{equation}
\EU_{2n+1}((R,\Delta),(I^{k},\Omega_{\min}^{I^{k}}\+\Omega\circ I^{k-1}))\leq H \leq \CU_{2n+1}((R,\Delta),(I,\Omega))
\end{equation}
where $k=\frac{12^d-1}{11}$ if $n=3$ resp. $k=\frac{10^d-1}{9}$ if $n\geq 4$, see Theorems \ref{thmm3} and \ref{thmm6}. For the even-dimensional unitary groups this improves the result obtained by You in \cite{you}. Moreover, we obtain our result (2) only by straightforward computations, without using localisation.

The rest of the paper is organised as follows. In Section 2, we recall some standard notation that is used throughout the paper. In Section 3, we recall the definitions of the groups $\U_{2n+1}(R,\Delta)$ and some important subgroups. In Section 4, we define the lower and upper level of a subgroup $H\leq \U_{2n+1}(R,\Delta)$. In Section 5, we prove our main results for the groups $\U_{7}(R,\Delta)$, namely Theorems \ref{thmm1}, \ref{thmm2} and \ref{thmm3}. In Section 6, we prove our main results for the groups $\U_{2n+1}(R,\Delta),n\geq 4$, namely Theorems \ref{thmm4}, \ref{thmm5} and \ref{thmm6}.

\section{Notation}
$\mathbb{N}$ denotes the set of all positive integers. If $G$ is a group and $g,h\in G$, we let $g^h:=h^{-1}gh$, $^hg:=hgh^{-1}$ and $[g,h]:=ghg^{-1}h^{-1}$. By a ring we mean an associative ring with $1$ such that $1\neq 0$. By an ideal we mean a two-sided ideal. If $n\in \N$ and $R$ is a ring, then the set of all $n\times n$ matrices with entries in $R$ is denoted by $M_n(R)$. If $a\in M_{n}(R)$, we denote the entry of $a$ at position $(i,j)$ by $a_{ij}$, the $i$-th row of $a$ by $a_{i*}$ and the $j$-th column of $a$ by $a_{*j}$. The group of all invertible matrices in $M_{n}(R)$ is denoted by $\GL_n(R)$ and the identity element of $\GL_n(R)$ by $e$. If $a\in \GL_n(R)$, then the entry of $a^{-1}$ at position $(i,j)$ is denoted by $a'_{ij}$, the $i$-th row of $a^{-1}$ by $a'_{i*}$ and the $j$-th column of $a^{-1}$ by $a'_{*j}$. Furthermore, we denote by $^n\!R$ the set of all row vectors of length $n$ with entries in $R$ and by $R^n$ the set of all column vectors of length $n$ with entries in $R$. We consider $^n\!R$ as left $R$-module and $R^n$ as right $R$-module.

\section{Odd-dimensional unitary groups}
We describe Hermitian form rings $(R,\Delta)$ and odd form ideals $(I,\Omega)$ first, then the odd-dimensional unitary group $\U_{2n+1}(R,\Delta)$ and its elementary subgroup $\EU_{2n+1}(R,\Delta)$ over a Hermitian form ring $(R,\Delta)$. For an odd form ideal $(I,\Omega)$, we recall the definitions of the following subgroups of $\U_{2n+1}(R,\Delta)$; the preelementary subgroup $\EU_{2n+1}(I, \Omega)$ of level $(I,\Omega)$, the elementary subgroup $\EU_{2n+1}((R,\Lambda),(I,\Omega))$ of level $(I,\Omega)$, the principal congruence subgroup $\U_{2n+1}((R,\Lambda),(I,\Omega))$ of level $(I,\Omega)$, the normalised principal congruence subgroup $\NU_{2n+1}((R,\Lambda),(I,\Omega))$ of level $(I,\Omega)$, and the full congruence subgroup $\CU_{2n+1}((R,\Lambda),(I,\Omega))$ of level $(I,\Omega)$.

\subsection{Hermitian form rings and odd form ideals}\label{sec 3.1}
First we recall the definitions of a ring with involution with symmetry and a Hermitian ring.
\begin{definition}
Let $R$ be a ring and 
\begin{align*}
\bar{}:R&\rightarrow R\\
x&\mapsto \bar{x}
\end{align*}
an anti-isomorphism of $R$ (i.e. $\bar{}~$ is bijective, $\overline{x+y}=\bar x+\bar y$, $\overline{xy}=\bar y\bar x$ for any $x,y\in R$ and $\bar 1=1$). Furthermore, let $\lambda\in R$ such that $\bar{\bar x}=\lambda x\bar\lambda$ for any $x\in R$. Then $\lambda$ is called a {\it symmetry} for $~\bar{}~$, the pair $(~\bar{}~,\lambda)$ an {\it involution with symmetry} and the triple $(R,~\bar{}~,\lambda)$ a {\it ring with involution with symmetry}. A subset $A\subseteq R$ is called {\it involution invariant} iff $\bar x\in A$ for any $x\in A$. A {\it Hermitian ring} is a quadruple $(R,~\bar{}~,\lambda,\mu )$ where $(R,~\bar{}~,\lambda)$ is a ring with involution with symmetry and $\mu \in R$ is a ring element such that $\mu =\bar\mu \lambda$ .
\end{definition}
\begin{remark}\label{25}
Let $(R,~\bar{}~,\lambda,\mu )$ be a Hermitian ring.
\begin{enumerate}[(a)]
\item It is easy to show that $\bar\lambda=\lambda^{-1}$.
\item The map
\begin{align*}
\b{}:R&\rightarrow R\\
x&\mapsto \b{x}:=\bar\lambda \bar x\lambda
\end{align*}
is the inverse map of $~\bar{}~$. One checks easily that $(R,~\b{}~,\b{$\lambda$},\b{$\mu $})$ is a Hermitian ring.
\end{enumerate}
\end{remark}

Next we recall the definition of an $R^{\bullet}$-module.
\begin{definition}
If $R$ is a ring, let $R^\bullet$ denote the underlying set of the ring equipped with the  
multiplication of the ring, but not the addition of the ring. A {\it (right) $R^{\bullet}$-module} is a not 
necessarily abelian group $(G,\+)$ equipped with a map 
\begin{align*}
\circ: G\times R^{\bullet}&\rightarrow G\\
(a,x) &\mapsto a\circ x
\end{align*}
such that the following holds:
\begin{enumerate}[(i)]
\item $a\circ 0=0$ for any $a\in G$,
\item $a\circ 1=a$ for any $a\in G$,
\item $(a\circ x)\circ y=a\circ (xy)$ for any $a\in G$ and $x,y\in R$ and
\item $(a\+ b)\circ x=(a\circ x)\+(b\circ x)$ for any $a,b\in G$ and $x\in R$.
\end{enumerate}
Let $G$ and $G'$ be $R^{\bullet}$-modules. A group homomorphism $f:G\rightarrow G'$ satisfying $f(a\circ x)=f(a)\circ x$ for any $a\in G$ and $x\in R$ is called a {\it  homomorphism of $R^{\bullet}$-modules}. A subgroup $H$ of $G$ which is $\circ$-stable (i.e. $a\circ x\in H$ for any $a\in H$ and $x\in R$) is called an {\it $R^{\bullet}$-submodule}. Moreover, if $A\subseteq G$ and $B\subseteq R$, we denote by $A\circ B$ the subgroup of $G$ generated by $\{a\circ b\mid a\in A,b\in B\}$. We treat $\circ$ as an operator with higher priority than $\+$.
\end{definition}

An important example of an $R^{\bullet}$-module is the Heisenberg group, which we define next. 

\begin{definition}\label{27}
Let $(R,~\bar{}~,\lambda,\mu )$ be a Hermitian ring. Define the map.
\begin{align*}
\+: (R\times R)\times (R\times R) &\rightarrow R\times R\\
((x_1,y_1),(x_2,y_2))&\mapsto (x_1,y_1)\+ (x_2,y_2):=(x_1+x_2,y_1+y_2-\bar x_1\mu  x_2).
\end{align*}
Then $(R\times R,\+)$ is a group, which we call the {\it Heisenberg group} and denote by $\h$. Equipped with the map
\begin{align*}
\circ:(R\times R)\times R^{\bullet}&\rightarrow R\times R\\
((x,y),a)&\mapsto (x,y)\circ a:=(xa,\bar aya)
\end{align*}
$\h$ becomes an $R^{\bullet}$-module. 
\end{definition}
\begin{remark}
We denote the inverse of an element $(x,y)\in \h$ by $\minus(x,y)$. One checks easily that $\minus(x,y)=(-x,-y-\bar x\mu  x)$ for any $(x,y)\in \h$.
\end{remark}

In order to define the odd-dimensional unitary groups we need the notion of a Hermitian form ring.
\begin{definition}
Let $(R,~\bar{}~,\lambda,\mu )$ be a Hermitian ring. Let $(R,+)$ have the $R^{\bullet}$-module structure defined by $x\circ a = \bar{a}xa$. Define the {\it trace map}
\begin{align*}
tr:\h&\rightarrow R\\
(x,y)&\mapsto \bar x\mu  x+y+\bar y\lambda.
\end{align*}
One checks easily that $\tr$ is a homomorphism of $R^{\bullet}$-modules. Set \[\Delta_{\min}:=\{(0,x-\overline{x}\lambda)\mid x\in R\}\] and \[\Delta_{\max}:=\ker(\tr).\] An $R^{\bullet}$-submodule $\Delta$ of $\h$ lying between $\Delta_{\min}$ and $\Delta_{\max}$ is called an {\it odd form parameter} for $(R,~\bar{}~,\lambda,\mu )$. Since $\Delta_{\min}$ and $\Delta_{\max}$ are $R^{\bullet}$-submodules of $\h$, they are respectively the smallest and largest odd form parameters. A pair $((R,~\bar{}~,\lambda,\mu ),\Delta)$ is called a {\it Hermitian form ring}. We shall usually abbreviate it by $(R,\Delta)$. 
\end{definition}
 
Next we define an odd form ideal of a Hermitian form ring. 
\begin{definition}
Let $(R,\Delta)$ be a Hermitian form ring and $I$ an involution invariant ideal of $R$. Set $J(\Delta):=\{y\in R\mid\exists z\in R:(y,z)\in \Delta\}$ and $\tilde I:=\{x\in R\mid\overline{J(\Delta)}\mu  x\subseteq I\}$. Obviously $\tilde I$ and $J(\Delta)$ are right ideals of $R$ and $I\subseteq \tilde I$. Moreover, set \[\Omega^I_{\min}:=\{(0,x-\bar x\lambda)\mid x\in I\}\+ \Delta\circ I\] and \[\Omega^I_{\max}:=\Delta\cap (\tilde I\times  I).\]
An $R^{\bullet}$-submodule $\Omega$ of $\h$ lying between $\Omega^I_{\min}$ and $\Omega^I_{\max}$ is called a {\it relative odd form parameter of level $I$}. Since $\Omega^I_{\min}$ and $\Omega^I_{\max}$ are $R^{\bullet}$-submodules of $\h$, they are respectively the smallest and the largest relative odd form parameters of level $I$. If $\Omega$ is a relative odd form parameter of level $I$, then $(I,\Omega)$ is called an {\it odd form ideal} of $(R,\Delta)$.
\end{definition}

\begin{definition}\label{defpq}
Let $(R,\Delta)$ be a Hermitian form ring where $R$ is commutative. Let $(I,\Omega)$ be an odd form ideal of $(R,\Delta)$ and $J$ an involution invariant ideal of $R$. We define  $(I,\Omega)\ast J$ as the odd form ideal $(IJ,\Omega_{\min}^{IJ}\+\Omega\circ J)$. Furthermore, we define $(I,\Omega):J$ as the odd form ideal $(I: J,\Omega_{\min}^{I:J}\+\{\alpha\in\Omega_{\max}^{I:J}\mid \alpha\circ J\subseteq \Omega\})$ where $I:J=\{x\in R\mid xJ\subseteq I\}$ is the usual quotient of ideals.
\end{definition}

\subsection{The odd-dimensional unitary group}
Let $(R,\Delta)$ be a Hermitian form ring and $n\in \mathbb{N}$. Set $M:=R^{2n+1}$. We use the following indexing for the elements of the standard basis of $M$: $(e_1,\dots,e_n,e_0,e_{-n},\dots,e_{-1})$.
That means that $e_i$ is the column whose $i$-th coordinate is one and all the other coordinates are zero if $1 \leq i\leq n$, the column whose $(n+1)$-th coordinate is one and all the other coordinates are zero if $i=0$, and the column whose $(2n+2+i)$-th coordinate is one and all the other coordinates are zero if $-n\leq i \leq -1$. If $u\in M$, then we call $(u_1,\dots,u_n,u_{-n},\dots,$ $u_{-1})^t\in R^{2n}$ the {\it hyperbolic part} of $u$ and denote it by $u_{\hb}$. We set $u^*:=\bar u^t$ and $u_{\hb}^*:=\bar u_{\hb}^t$. Moreover, we define the maps
\begin{align*}
B:M\times M&\rightarrow R\\
(u,v)&\mapsto u^*\begin{pmatrix} 0& 0 & \pi\\0&\mu &0\\ \pi\lambda &0 &0 \end{pmatrix}v=\sum\limits_{i=1}^{n}\bar u_i v_{-i}+\bar u_0\mu  v_0+\sum\limits_{i=-n}^{-1}\bar u_{i}\lambda v_{-i}
\end{align*}
and 
\begin{align*}
Q:M&\rightarrow \h\\
u&\mapsto (Q_1(u),Q_2(u)):=(u_0,u_{\hb}^*\begin{pmatrix} 0&\pi\\0&0 \end{pmatrix}u_{\hb})=(u_0,\sum\limits_{i=1}^{n}\bar u_i u_{-i})			
\end{align*}
where $\pi\in M_n(R)$ denotes the matrix with ones on the skew diagonal and zeros elsewhere.

\begin{lemma}
~\\
\vspace{-0.6cm}
\begin{enumerate}[(i)]
\item $B$ is a \textnormal{$\lambda$-Hermitian form}, i.e. $B$ is biadditive, $B(ux,vy)=\bar x B(u,v) y~\forall u,v\in M,x,y\in R$ and $B(u,v)=\overline{B(v,u)}\lambda~\forall u,v\in M$.
\item $Q(ux)=Q(u)\circ x~\forall u\in M, x\in R$, $Q(u+v)\equiv Q(u)\+ Q(v)\+(0,B(u,v))\bmod \Delta_{\min}~\forall u,v\in M$ and $\tr(Q(u))=B(u,u)~\forall u\in M$.
\end{enumerate}
\end{lemma}
\begin{proof}
{Straightforward computation.}
\end{proof}

\begin{definition}
The group $\U_{2n+1}(R,\Delta):=$
\begin{align*}
\{\sigma\in \GL_{2n+1}(R)\mid B(\sigma u,\sigma v)=B(u,v)\land Q(\sigma u)\equiv Q(u)\bmod \Delta~\forall u,v\in M\}
\end{align*}
is called the {\it odd-dimensional unitary group}.
\end{definition}


\begin{remark}\label{34}
The groups $\U_{2n+1}(R,\Delta)$ include as special cases the even-dimensional unitary groups $\U_{2n}(R,\Lambda)$ and all classical Chevalley groups. On the other hand, the groups $\U_{2n+1}(R,\Delta)$ are embraced by Petrov's odd unitary groups $\U_{2l}(R,\mathfrak{L})$. For details see \cite[Remark 14(c) and Example 15]{bak-preusser}. 
\end{remark}

\begin{definition}
We define the sets $\Theta_+:=\{1,\dots,n\}$, $\Theta_-:=\{-n,\dots,-1\}$, $\Theta:=\Theta_+\cup\Theta_-\cup\{0\}$ and $\Theta_{\hb}:=\Theta\setminus \{0\}$. Moreover, we define the map 
\begin{align*}\epsilon:\Theta_{\hb} &\rightarrow\{\pm 1\}\\
i&\mapsto\begin{cases} 1, &\mbox{if } i\in\Theta_+, \\ 
-1, & \mbox{if } i\in\Theta_-. \end{cases}
\end{align*}
\end{definition}

\begin{remark}
We will sometimes use expressions of the form $\sum\limits_{i\in A}f(i)$ (resp. $\prod\limits_{i\in A}f(i)$) where $A\subseteq \Theta$ is a subset and $f:A\rightarrow X$ is a map where $X$ is a set with a fixed addition $+$ (resp. multiplication $\cdot$). In such expressions we assume that the order of the summands (resp. factors) corresponds to the strict total order $1\prec \dots\prec n\prec 0 \prec -n \dots \prec -1$ on $\Theta$.
\end{remark}

\begin{lemma}[{\cite[Lemma 17]{bak-preusser}}]\label{36}
Let $\sigma\in \GL_{2n+1}(R)$. Then $\sigma\in \U_{2n+1}(R,\Delta)$ iff Conditions (i) and (ii) below are satisfied. 
\begin{enumerate}[(i)]
\item \begin{align*}
       \sigma'_{ij}&=\lambda^{-(\epsilon(i)+1)/2}\bar\sigma_{-j,-i}\lambda^{(\epsilon(j)+1)/2}~\forall i,j\in\Theta_{\hb},\\
       \mu \sigma'_{0j}&=\bar\sigma_{-j,0}\lambda^{(\epsilon(j)+1)/2}~\forall j\in\Theta_{\hb},\\
       \sigma'_{i0}&=\lambda^{-(\epsilon(i)+1)/2}\bar\sigma_{0,-i}\mu ~\forall i\in\Theta_{\hb} \text { and}\\
       \mu \sigma'_{00}&=\bar\sigma_{00}\mu .
      \end{align*}
\item \begin{align*}
Q(\sigma_{*j})\equiv (\delta_{0j},0)\bmod \Delta ~\forall j\in \Theta.
\end{align*} 
\end{enumerate}
\end{lemma}

\subsection{The polarity map}
\begin{definition}
The map
\begin{align*}
\widetilde{}: M&\longrightarrow M^*\\
u&\longmapsto \begin{pmatrix} \bar u_{-1}\lambda&\dots&\bar u_{-n}\lambda&\bar u_0\mu&\bar u_{n}&\dots&\bar u_1\end{pmatrix}
\end{align*}
where $M^*={}^{2n+1}\!R$ is called the {\it polarity map}. Clearly $~\widetilde{}~$ is {\it involutary linear}, i.e. $\widetilde{u+v}=\tilde u+\tilde v$ and $\widetilde{ux}=\bar x\tilde u$ for any $u,v\in M$ and $x\in R$.
\end{definition}
\begin{lemma}\label{38}
If $\sigma\in \U_{2n+1}(R,\Delta)$ and $u\in M$, then $\widetilde{\sigma u}=\tilde u\sigma^{-1}$.
\end{lemma}
\begin{proof}
Follows from Lemma \ref{36}.
\end{proof}

\subsection{The elementary subgroup}
We introduce the following notation. Let $(R,~\b{}~,\b{$\lambda$},\b{$\mu $})$ be the Hermitian ring defined in Remark \ref{25}(b) and $\h^{-1}$ the corresponding Heisenberg group. Note that the underlying set of both $\h$ and $\h^{-1}$ is $R\times R$. We denote the group operation (resp. scalar multiplication) of $\h$ by $\+_1$ (resp. $\circ_1$) and the group operation (resp. scalar multiplication) of $\h^{-1}$ by $\+_{-1}$ (resp. $\circ_{-1}$). Furthermore, we set $\Delta^1:=\Delta$ and $\Delta^{-1}:=\{(x,y)\in R\times R\mid(x,\bar y)\in \Delta\}$. One checks easily that $((R,~\b{}~,\b{$\lambda$},\b{$\mu $}),\Delta^{-1})$ is a Hermitian form ring. Analogously, if $(I,\Omega)$ is an odd form ideal of $(R,\Delta)$, we set $\Omega^1:=\Omega$ and $\Omega^{-1}:=\{(x,y)\in R\times R\mid(x,\bar y)\in \Omega\}$. One checks easily that $(I,\Omega^{-1})$ is an odd form ideal of $(R,\Delta^{-1})$. 
 
If $i,j\in \Theta$, let $e^{ij}$ denote the matrix in $M_{2n+1}(R)$ with $1$ in the $(i,j)$-th position and $0$ in all other positions.
\begin{definition}
If $i,j\in \Theta_{\hb}$, $i\neq\pm j$ and $x\in R$, the element  \[T_{ij}(x):=e+xe^{ij}-\lambda^{(\epsilon(j)-1)/2}\bar x\lambda^{(1-\epsilon(i))/2}e^{-j,-i}\] of $\U_{2n+1}(R,\Delta)$ is called an {\it (elementary) short root transvection}. 
If $i\in \Theta_{\hb}$ and $(x,y)\in \Delta^{-\epsilon(i)}$, the element \[T_{i}(x,y):=e+xe^{0,-i}-\lambda^{-(1+\epsilon(i))/2}\bar x\mu e^{i0}+ye^{i,-i}\] of $\U_{2n+1}(R,\Delta)$ is called an {\it (elementary) extra short root transvection}. The extra short root transvections of the kind \[T_{i}(0,y)=e+ye^{i,-i}\] are called {\it (elementary) long root transvections}. If an element of $\U_{2n+1}(R,\Delta)$ is a short or extra short root transvection, then it is called an {\it elementary transvection}. The subgroup of $\U_{2n+1}(R,\Delta)$ generated by all elementary transvections is called the {\it elementary subgroup} and is denoted by $\EU_{2n+1}(R,\Delta)$. 
\end{definition}

\begin{lemma}[{\cite[Lemma 20]{bak-preusser}}]\label{39}
The following relations hold for elementary transvections.
\begin{align*}
&T_{ij}(x)=T_{-j,-i}(-\lambda^{(\epsilon(j)-1)/2}\bar x\lambda^{(1-\epsilon(i))/2}), \tag{S1}\\
&T_{ij}(x)T_{ij}(y)=T_{ij}(x+y), \tag{S2}\\
&[T_{ij}(x),T_{kl}(y)]=e \text{ if } k\neq j,-i \text{ and } l\neq i,-j, \tag{S3}\\
&[T_{ij}(x),T_{jk}(y)]=T_{ik}(xy) \text{ if } i\neq\pm k, \tag{S4}\\
&[T_{ij}(x),T_{j,-i}(y)]=T_{i}(0,xy-\lambda^{(-1-\epsilon(i))/2}\bar y\bar x\lambda^{(1-\epsilon(i))/2}), \tag{S5}\\
&T_{i}(x_1,y_1)T_{i}(x_2,y_2)=T_{i}((x_1,y_1)\+_{-\epsilon(i)}(x_2,y_2)), \tag{E1}\\
&[T_{i}(x_1,y_1),T_{j}(x_2,y_2)]=T_{i,-j}(-\lambda^{-(1+\epsilon(i))/2}\bar x_1\mu x_2) \text{ if } i\neq\pm j, \tag{E2}\\
&[T_{i}(x_1,y_1),T_{i}(x_2,y_2)]=T_{i}(0,-\lambda^{-(1+\epsilon(i))/2}(\bar x_1\mu x_2-\bar x_2\mu x_1)), \tag{E3}\\
&[T_{ij}(x),T_{k}(y,z)]=e \text{ if } k\neq j,-i \text{ and} \tag{SE1}\\
&[T_{ij}(x),T_{j}(y,z)]=T_{j,-i}(z\lambda^{(\epsilon(j)-1)/2}\bar x\lambda^{(1-\epsilon(i))/2})\cdot\\
&\hspace{3.4cm}\cdot T_{i}(y\lambda^{(\epsilon(j)-1)/2}\bar x\lambda^{(1-\epsilon(i))/2},xz\lambda^{(\epsilon(j)-1)/2}\bar x\lambda^{(1-\epsilon(i))/2}).\tag{SE2}\\
\end{align*}
\end{lemma}

\subsection{Relative subgroups}
In this subsection $(I,\Omega)$ denotes an odd form ideal of $(R,\Delta)$.
\begin{definition}
A short root transvection $T_{ij}(x)$ is called {\it $(I,\Omega)$-elementary} if $x\in I$. An extra short root transvection $T_{i}(x,y)$ is called {\it $(I,\Omega)$-elementary} if $(x,y)\in \Omega^{-\epsilon(i)}$. 
The subgroup $\EU_{2n+1}(I,\Omega)$ of $\EU_{2n+1}(R,\Delta)$ generated by the $(I,\Omega)$-elementary trans\-vections is called the {\it preelementary subgroup of level $(I,\Omega)$}. Its normal closure $\EU_{2n+1}((R,\Delta),(I,\Omega))$ in $\EU_{2n+1}(R,\Delta)$ is called the {\it elementary subgroup of level $(I,\Omega)$}.
\end{definition}

If $\sigma\in M_{2n+1}(R)$, we call the matrix $(\sigma_{ij})_{i,j\in\Theta_{\hb}}\in M_{2n}(R)$ the {\it hyperbolic part} of $\sigma$ and denote it by $\sigma_{\hb}$. Furthermore, we define the submodule $M(R,\Delta):=\{u\in M\mid u_0\in J(\Delta)\}$ of $M$. 
\begin{definition}
The subgroup $\U_{2n+1}((R,\Delta),(I,\Omega)):=$
\[\{\sigma\in \U_{2n+1}(R,\Delta)\mid\sigma_{\hb}\equiv e_{\hb}\bmod  I\text{ and }Q(\sigma u)\equiv Q(u)\bmod \Omega~\forall u\in M(R,\Delta)\}\]
of $\U_{2n+1}(R,\Delta)$ is called {\it the principal congruence subgroup of level $(I,\Omega)$}.
\end{definition}

\begin{lemma}[{\cite[Lemma 28]{bak-preusser}}]\label{46}
Let $\sigma\in \U_{2n+1}(R,\Delta)$. Then $\sigma\in \U_{2n+1}((R,\Delta),(I,\Omega))$ iff Conditions (i) and (ii) below are satisfied. 
\begin{enumerate}[(i)]
\item $\sigma_{\hb}\equiv e_{\hb}\bmod  I$.
\item $Q(\sigma_{*j})\in\Omega~\forall j\in \Theta_{\hb}$ and $(Q(\sigma_{*0})\minus (1,0))\circ a\in\Omega~\forall a\in J(\Delta)$.
\end{enumerate}
\end{lemma}


\begin{definition}
The subgroup $\NU_{2n+1}((R,\Delta),(I,\Omega)):=$
\[\Normaliser_{\U_{2n+1}(R,\Delta)}(\U_{2n+1}((R,\Delta),(I,\Omega)))\]
of $\U_{2n+1}(R,\Delta)$ is called the {\it normalised principal congruence subgroup of level $(I,\Omega)$}.
\end{definition}


\begin{definition}
The subgroup $\CU_{2n+1}((R,\Delta),(I,\Omega)):=$
\[ \{\sigma\in \NU_{2n+1}((R,\Delta),(I,\Omega))\mid [\sigma,\EU_{2n+1}(R,\Delta)]\leq \U_{2n+1}((R,\Delta),(I,\Omega))\}\]
of $\U_{2n+1}(R,\Delta)$ is called the {\it full congruence subgroup of level $(I,\Omega)$}.
\end{definition}


\section{The lower and upper level of a subgroup of an odd-dimensional unitary group}
In this section $n$ denotes an integer greater than or equal to $3$ and $(R,\Delta)$ a Hermitian form ring where $R$ is commutative. We will define the lower and the upper level of a subgroup of $\U_{2n+1}(R,\Delta)$.

\begin{definition}\label{defll}
Let $H$ be a subgroup of $\U_{2n+1}(R,\Delta)$. Set
\begin{align*}
I:=\{x\in R\mid &\hspace{0.5cm}T_{ij}(xr)^{\tau}\in H~\forall i,j\in\Theta_{hb},i\neq\pm j,r\in R,\tau\in \EU_{2n+1}(R,\Delta)\\
&\land T_{i}(0,xr-\overline{xr}\lambda^{-\epsilon(i)})^{\tau}\in H~\forall i\in\Theta_{hb},r\in R,\tau\in \EU_{2n+1}(R,\Delta)\\
&\land T_{i}(\alpha\circ xr)^{\tau}\in H~\forall i\in\Theta_{hb},\alpha\in \Delta^{-\epsilon(i)},r\in R,\tau\in \EU_{2n+1}(R,\Delta)\\
&\land T_{i}(\alpha\circ \bar xr)^{\tau}\in H~\forall i\in\Theta_{hb},\alpha\in \Delta^{-\epsilon(i)},r\in R,\tau\in \EU_{2n+1}(R,\Delta)\}
\end{align*}
and 
\begin{align*}
\Omega:=\Omega_{\min}^I\+\{(y,z)\in \Omega_{\max}^I\mid&\hspace{0.5cm} T_i((y,z)\circ r)^{\tau}\in H~\forall i\in\Theta_{-},r\in R,\tau\in \EU_{2n+1}(R,\Delta)\\
&\land T_i((y,\bar z)\circ r)^{\tau}\in H~\forall i\in\Theta_{+},r\in R,\tau\in \EU_{2n+1}(R,\Delta)\}.
\end{align*}
The odd form ideal $L(H):=(I,\Omega)$ is called the \emph{lower level} of $H$.
\end{definition}

If $H$ is a subgroup of $\U_{2n+1}(R,\Delta)$, then clearly $\EU_{2n+1}((R,\Delta),L(H))\leq H$ and $L(H)$ is the greatest odd form ideal with this property.

\begin{definition}\label{deful1}
Let $H$ be a subgroup of $\U_{2n+1}(R,\Delta)$. Set
\[Y:=\{\sigma_{ij},\sigma_{ii}-\sigma_{jj}, \sigma_{i0}J(\Delta),\overline{J(\Delta)}\mu\sigma_{0j},\overline{J(\Delta)}\mu(\sigma_{00}-\sigma_{jj})J(\Delta)\mid \sigma\in H, i,j\in \Theta_{\hb},i\neq j\}\]
and 
\[Z:=\{Q(\sigma_{*j}),(Q(\sigma_{*0})\minus(1,0))\circ y\+(y,z)\minus (y,z)\circ \sigma_{ii}\mid \sigma\in H, i,j\in \Theta_{\hb}, (y,z)\in \Delta
\}.\]
Let $I$ be the involution invariant ideal generated by $Y$ (i.e. the ideal generated by $Y\cup \bar Y$) and set $\Omega:=\Omega_{\min}^I\+
Z\circ R$. The odd form ideal $U(H):=(I,\Omega)$ is called the \emph{upper level} of $H$.
\end{definition}

We will show that if $H$ is a subgroup of $\U_{2n+1}(R,\Delta)$, then $H\leq\CU_{2n+1}((R,\Delta),$ $U(H))$ and $U(H)$ is the smallest odd form ideal with this property.

\begin{lemma}\label{lemul1}
Let $\sigma\in\U_{2n+1}(R,\Delta)$ and $(I,\Omega)$ be an odd form ideal of $(R,\Delta)$. Then $\sigma\in\NU_{2n+1}((R,\Delta),(I,\Omega))$ iff $(Q(\sigma_{*0})\minus(1,0))\circ x\in \Omega$ and $(Q(\sigma'_{*0})\minus(1,0))\circ x\in \Omega$ for any $x\in J(\Omega)$.
\end{lemma}
\begin{proof}
By \cite[Corollary 35]{bak-preusser}, we have $\sigma \in\NU_{2n+1}((R,\Delta),(I,\Omega))$ iff $\Omega={}^{\sigma}\Omega$ where ${}^{\sigma}\Omega=\{(Q(\sigma_{*0})\minus (1,0))\circ x\+(x,y)\mid (x,y)\in\Omega\}\+\Omega_{\min}^I$ (see \cite[Definition 30]{bak-preusser}). By the definition of equality of sets, we have $\Omega={}^{\sigma}\Omega$ iff $\Omega\subseteq{}^{\sigma}\Omega$ and ${}^{\sigma}\Omega\subseteq \Omega$. By \cite[Lemma 33]{bak-preusser}, the map $\U_{2n+1}(R,\Delta)\times \FP(I)\rightarrow \FP(I),~(\tau,\Sigma)\mapsto {}^{\tau}\Sigma$,
where $\FP(I)$ denotes the set of all relative odd form parameters for $I$, is a left group action. Clearly this action preserves inclusions. It follows that $\Omega\subseteq{}^{\sigma}\Omega$ iff ${}^{\sigma^{-1}}\Omega\subseteq\Omega$. Hence $\sigma \in\NU_{2n+1}((R,\Delta),(I,\Omega))$ iff ${}^{\sigma^{-1}}\Omega\subseteq\Omega$ and ${}^{\sigma}\Omega\subseteq \Omega$. The assertion of the lemma follows.
\end{proof}

\begin{corollary}\label{corul1}
Let $H$ be a subgroup of $\U_{2n+1}(R,\Delta)$. Then $H\leq \NU_{2n+1}((R,\Delta),U(H))$.
\end{corollary}
\begin{proof}
Write $U(H)=(I,\Omega)$ and let $\sigma\in H$. By the previous lemma, it suffices to show that $(Q(\sigma_{*0})\minus(1,0))\circ x\in \Omega$ for any $x\in J(\Omega)$. Let $x\in J(\Omega)$. Then $(x,y)\in \Omega$ for some $y\in R$. Hence $(Q(\sigma_{*0})\minus(1,0))\circ x+(x,y)\minus (x,y)\circ \sigma_{11}\in Z\subseteq \Omega$ where $Z$ is defined as in Definition \ref{deful1}. It clearly follows that $(Q(\sigma_{*0})\minus(1,0))\circ x\in\Omega$.
\end{proof}

\begin{lemma}\label{last}
Let $(a,b)\in \Delta^{k}$ and $x_1,\dots,x_m\in R$ where $k\in\{\pm 1\}$ and $m\in\N$. Then
\begin{align*}
(a,b)\circ \sum\limits_{i=1}^mx_i=(a,b)\circ x_1\+_k\dots \+_k (a,b)\circ x_m\+_k(0,\sum\limits_{\substack{i,j=1,\\i>j}}^m \bar x_ibx_j-\overline{\bar x_ibx_j}\lambda^{k}).
\end{align*}
\end{lemma}
\begin{proof}
Straightforward computation.
\end{proof}

\begin{lemma}\label{lemul2}
Let $(I,\Omega)$ be an odd form ideal. Suppose $\sigma\in \U_{2n+1}(R,\Delta)$ satisfies Conditions (i)-(v) in Lemma \ref{lemul4}. Then $(a,b)\equiv (a,b)\circ \sigma'_{ii}\sigma_{ii} \bmod \Omega$ for any $(a,b)\in \Delta$ and $i\in \Theta_{\hb}$.
\end{lemma}
\begin{proof}
By the previous lemma we have 
\begin{align}
&(a,b)\notag\\
=&(a,b)\circ(\sum\limits_{s=1}^{-1}\sigma'_{is}\sigma_{si})\notag\\
=&(\plus\limits_{s=1}^{-1} (a,b)\circ\sigma'_{is}\sigma_{si})\+(0,\sum\limits_{\substack{s,t=1,\\s\succ t}}^{-1} \overline{\sigma'_{is}\sigma_{si}}b\sigma'_{it}\sigma_{ti}-\overline{\overline{\sigma'_{is}\sigma_{si}}b\sigma'_{it}\sigma_{ti}}\lambda))
\end{align}
Since $\sigma$ satisfies Conditions (i) and (iv) in Lemma \ref{lemul4}, all the summands in (3) except $(a,b)\circ \sigma'_{ii}\sigma_{ii}$ are contained in $\Omega^I_{\min}$. 
\end{proof}

\begin{lemma}\label{lemul3}
Let $(I,\Omega)$ be an odd form ideal. Suppose $\sigma\in \U_{2n+1}(R,\Delta)$ satisfies Conditions (i)-(v) in Lemma \ref{lemul4}. Then $(0,-\bar\sigma'_{-i,-i}\bar y \bar\sigma_{00}\mu y+ \overline{\bar\sigma'_{-i,-i}\bar y \bar\sigma_{00}\mu y}\lambda)\in\Omega$ for any $i\in \Theta_{hb}$ and $y\in J(\Delta)$.
\end{lemma}
\begin{proof}

First we note that by \cite[Lemma 6.30]{preusser_sct2}, $\sigma^{-1}$ also satisfies Conditions (i)-(v) in Lemma \ref{lemul4}. Clearly
\begin{align*}
&(0,-\bar\sigma'_{-i,-i}\bar y \bar\sigma_{00}\mu y+ \overline{\bar\sigma'_{-i,-i}\bar y \bar\sigma_{00}\mu y}\lambda)\\
=&(0,-\sigma_{ii}\bar y \mu \sigma'_{00}y+ \overline{\sigma_{ii}\bar y \mu \sigma'_{00}y}\lambda)\\
=&\underbrace{(0,-\sigma_{ii}\bar y \mu (\sigma'_{00}-\sigma'_{ii})y+ \overline{\sigma_{ii}\bar y \mu (\sigma'_{00}-\sigma'_{ii})y}\lambda)}_{X:=}\+\underbrace{(0,-\sigma_{ii}\bar y \mu \sigma'_{ii}y+ \overline{\sigma_{ii}\bar y \mu \sigma'_{ii}y}\lambda)}_{Y:=},
\end{align*}
the first equality by Lemma \ref{36}. Since $\sigma'$ satisfies Condition (v), $X$ lies in $\Omega_{\min}^I$. On the other hand, $\sigma_{ii}\sigma_{ii}'=1-\sum_{j\neq i}\sigma_{ij}\sigma_{ji}'\equiv 1 \bmod I$ since $\sigma$ satisfies satisfies Conditions (i) and (iii). Hence
\[Y\equiv (0,-\bar y \mu y+ \overline{\bar y \mu y}\lambda)=0 \bmod \Omega_{\min}^I\]
since $\mu=\bar\mu\lambda$.
\end{proof}

\begin{lemma}\label{lemul4}
Let $\sigma\in \U_{2n+1}(R,\Delta)$ and $(I,\Omega)$ an odd form ideal. Then $[\sigma,\EU_{2n+1}(R,$ $\Delta)]\leq\U_{2n+1}((R,\Delta),(I,\Omega))$ iff 
\begin{enumerate}[(i)]
\item $\sigma_{ij}\in I ~\forall i\neq j\in \Theta_{\hb}$, 
\item $\sigma_{ii}-\sigma_{jj}\in I ~\forall i,j\in \Theta_{\hb}$,
\item $\sigma_{i0}J(\Delta)\in I~\forall i\in \Theta_{\hb}$,
\item $\overline{J(\Delta)}\mu\sigma_{0j}\in I~\forall j\in \Theta_{\hb}$,
\item $\overline{J(\Delta)}\mu(\sigma_{00}-\sigma_{jj})J(\Delta)\in I ~\forall j\in \Theta_{\hb}$,
\item $Q(\sigma_{*j})\in\Omega~\forall j\in \Theta_{\hb}$ and 
\item $(Q(\sigma_{*0})\minus(1,0))\circ y\+(y,z)\minus (y,z)\circ \sigma_{ii}\in \Omega~\forall i\in \Theta_{\hb}, (y,z)\in \Delta$.
\end{enumerate}
\end{lemma}
\begin{proof}
$(\Rightarrow)$ Suppose that $[\sigma,\EU_{2n+1}(R,\Delta)]\leq\U_{2n+1}((R,\Delta),(I,\Omega))$. Then $[\sigma,\EU_{2n+1}($ $R,\Delta)]\leq\U_{2n+1}((R,\Delta),(I,\Omega^I_{\max}))$ and hence $\sigma\in  \CU_{2n+1}((R,\Delta),(I,\Omega^I_{\max}))$ (note that $\NU_{2n+1}((R,\Delta),(I,\Omega^I_{\max}))=U_{2n+1}(R,\Delta)$, see \cite[Remark 26]{bak-preusser}). It follows from \cite[Lemma 6.30]{preusser_sct2} that Conditions (i)-(v) above hold. By analysing \cite[Lemma 63]{bak-preusser} we obtain $Q(\sigma_{*i})\circ \sigma'_{jj}\in \Omega$ and $(Q(\sigma_{*0})\minus(1,0))\circ y\sigma'_{ii}\+(y,z)\circ \sigma'_{ii}\minus (y,z)\in \Omega~$ for any $i\neq \pm j\in \Theta_{\hb}$ and $(y,z)\in \Delta$. There are the following misprints in \cite[Lemma 63]{bak-preusser}, all on page 2866: $a_{-k}$ should be replaced by $a_{-i}$ (1 occurence), $\sigma'_{-1,-1}$ should be replaced by $\sigma'_{-i,-i}$ (4 occurences) and $(0,\bar\sigma'_{-i,-i}\bar y \bar\sigma_{00}\mu y- \overline{\bar\sigma'_{-i,-i}\bar y \bar\sigma_{00}\mu y}\lambda)$ should be replaced by $(0,-\bar\sigma'_{-i,-i}\bar y \bar\sigma_{00}\mu y+ \overline{\bar\sigma'_{-i,-i}\bar y \bar\sigma_{00}\mu y}\lambda)$ (1 occurence; note that $(0,-\bar\sigma'_{-i,-i}\bar y \bar\sigma_{00}\mu y+ \overline{\bar\sigma'_{-i,-i}\bar y \bar\sigma_{00}\mu y}\lambda)\in\Omega$ by Lemma \ref{lemul3}). It follows from Lemma \ref{lemul2} that $\sigma$ satisfies Conditions (vi) an (vii).\\
$(\Leftarrow)$ Now suppose that $\sigma$ satisfies Conditions (i)-(vii). We have to show that $[\sigma,\EU_{2n+1}$ $(R,\Delta)]\leq\U_{2n+1}((R,\Delta),(I,\Omega))$. Since $\U_{2n+1}((R,\Delta),(I,\Omega))$ is normalised by $\EU_{2n+1}($ $R,\Delta)$ (follows from \cite[Corollary 36]{bak-preusser}), it suffices to show that $[\sigma,\tau]\in\U_{2n+1}((R,\Delta),(I,$ $\Omega))$ for any elementary transvection $\tau$. Since $\sigma\in \CU_{2n+1}((R,\Delta),(I,\Omega^I_{\max}))$ by \cite[Lemma 6.30]{preusser_sct2}, we obtain that $[\sigma,\tau]$ satisfies Condition (i) in Lemma \ref{46}. It remains to show that $[\sigma,\tau]$ satisfies Condition (ii) in Lemma \ref{46}. But that follows from \cite[Lemma 63]{bak-preusser}.
\end{proof}

If $(I,\Omega)$ and $(J,\Sigma)$ are odd form parameters, then we write $(I,\Omega)\subseteq (J,\Sigma)$ if $I\subseteq J$ and $\Omega\subseteq \Sigma$.
\begin{proposition}\label{propul}
Let $H$ be a subgroup of $\U_{2n+1}(R,\Delta)$. Then $H\leq\CU_{2n+1}((R,\Delta),$ $U(H))$ and $U(H)$ is the smallest form odd ideal with this property.
\end{proposition}
\begin{proof}
It follows from Corollary \ref{corul1} and Lemma \ref{lemul4} that $H\leq\CU_{2n+1}((R,\Delta),U(H))$. Let $Y$ and $Z$ be defined as in Definition \ref{deful1}. If $(I,\Omega)$ is an odd form ideal such that $H\leq\CU_{2n+1}((R,\Delta),(I,\Omega))$, then $Y\subseteq I$ and $Z\subseteq \Omega$ by Lemma \ref{lemul4}. It follows that $U(H)\subseteq (I,\Omega)$.
\end{proof}

\begin{corollary}\label{corul3}
Let $H$ be a subgroup of $\U_{2n+1}(R,\Delta)$. Then $L(H)\subseteq U(H)$.
\end{corollary}
\begin{proof}
By Proposition \ref{propul} we have 
\[\EU_{2n+1}((R,\Delta),L(H))\leq H\leq\CU_{2n+1}((R,\Delta),U(H)).\]
It follows from Lemma \ref{lemul4} that $L(H)\subseteq U(H)$.
\end{proof}

\begin{corollary}\label{corul2}
Let $H$ be a subgroup of $\U_{2n+1}(R,\Delta)$ and $\tau\in \EU_{2n+1}(R,\Delta)$. Then $U(H)=U(H^\tau)$.
\end{corollary}
\begin{proof}
By Proposition \ref{propul} we have $H\leq\CU_{2n+1}((R,\Delta),U(H))$. It follows that $H^\tau\leq\CU_{2n+1}((R,\Delta),U(H))$ since $\EU_{2n+1}(R,\Delta)\leq \NU_{2n+1}((R,\Delta),U(H))$ by \cite[Corollary 36]{bak-preusser}. This implies that $U(H^{\tau})\subseteq U(H)$, again by Proposition \ref{propul}. Similarly one can show that $U(H)\subseteq U(H^{\tau})$.
\end{proof}

\section{The subnormal structure of the groups $\U_7(R,\Delta)$}

In this section $(R,\Delta)$ denotes a Hermitian form ring where $R$ is commutative.

\begin{lemma}\label{lemsub1}
Let $(I,\Omega)$ be an odd form ideal and $H$ a subgroup of $\U_{7}(R,\Delta)$ normalised by $\EU_{7}(I,\Omega)$. Suppose there is an $x\in R$, $r,s\in\Theta_{\hb}, r\neq \pm s$ and an $m\in\N$ such that $T_{r,\pm s}(xa)\in H$ for all $a\in I^m$. Then $T_{ij}(xa)\in H$ for all $i,j\in \Theta_{\hb}, i\neq\pm j$ and $a\in I^{m+3}$.
\end{lemma}
\begin{proof}
Choose a $t\in \Theta_{\hb}$ such that $t\neq \pm r, \pm s$. It follows from Relation (S4) in Lemma \ref{39} that
\begin{itemize}
\item $T_{r,\pm t}(xa),T_{\pm t ,\pm s}(xa)\in H$ for any $a\in I^{m+1}$,
\item $T_{\pm s ,\pm t}(xa),T_{\pm t ,\pm r}(xa),T_{-r ,\pm s}(xa)\in H$ for any $a\in I^{m+2}$ and
\item $T_{\pm s,\pm r}(xa),T_{-r ,\pm t}(xa)\in H$ for any $a\in I^{m+3}$.
\end{itemize}
The assertion of the lemma follows.
\end{proof}

We recall some notation introduced in \cite{preusser_decomp_almcom}. Let $G$ be a group and $(a_1,b_1), (a_2,b_2)\in G\times G$. If there is a $g\in G$ such that 
\[a_2=[a_1^{-1},g]\text{ and }b_2=[g,b_1],\]
then we write $(a_1,b_1)\xrightarrow{g} (a_2,b_2)$. If $(a_1,b_1),\dots, (a_{n+1},b_{n+1})\in G\times G$ and $g_1,\dots,g_{n}\in G$  such that 
\[(a_1,b_1) \xrightarrow{g_1} (a_2,b_2)\xrightarrow{g_2}\dots \xrightarrow{g_{n}} (a_{n+1},b_{n+1}),\]
then we write $(a_1,b_1) \xrightarrow{g_1,\dots,g_{n}}(a_{n+1},b_{n+1})$. 

If $H\leq G$, $g\in G$ and $h\in H$, then we call $g^h$ an {\it $H$-conjugate} of $g$.
\begin{lemma}[{\cite[Lemma 7]{preusser_decomp_almcom}}]\label{lemredux}
Let $G$ be a group and $(a_1,b_1),(a_2,b_2)\in G\times G$. If $(a_1,b_1)\xrightarrow{g_1,\dots,g_n}(a_2,b_2)$ for some $g_1,\dots,g_n\in G$, then $a_2b_2$ is a product of $2^n$ $H$-conjugates of $a_1b_1$ and $(a_1b_1)^{-1}$ where $H$ is the subgroup of $G$ generated by $\{a_1,g_1,\dots, g_n\}$.
\end{lemma}

\begin{lemma}\label{lemsub2}
Let $(I,\Omega)$ be an odd form ideal and $H$ a subgroup of $\U_{7}(R,\Delta)$ normalised by $\EU_{7}(I,\Omega)$. Let $\sigma\in H$ and $r,s,t
\in\Theta_{\hb}$ such that $r\neq \pm s$ and $t\neq\pm r,\pm s$. Then 
\begin{enumerate}[(i)]
\item $T_{ij}(\sigma_{s,-t}\bar\sigma_{sr}a)\in H$ for any $i,j\in \Theta_{\hb}, i\neq \pm j, a\in I^7$ and
\item $T_{ij}(\sigma_{rs}\bar\sigma_{rr}a)\in H$ for any $i,j\in \Theta_{\hb}, i\neq \pm j, a\in I^8$.
\end{enumerate}
\end{lemma}
\begin{proof}
Let $a_1,\dots,a_5\in I$. Set 
\begin{align*}
\tau_1:=&T_{rt}(\sigma_{ss}\bar\sigma_{sr}\bar a_1)T_{st}(-\sigma_{sr}\bar\sigma_{sr}\bar a_1)T_{r,-s}(-\lambda^{(\epsilon(t)-\epsilon(s))/2}\sigma_{s,-t}\bar\sigma_{sr}a_1)\cdot\\&\cdot T_{r}(0,\lambda^{(\epsilon(t)-\epsilon(r))/2}\sigma_{s,-t}\bar\sigma_{ss} a_1-\lambda^{(-\epsilon(t)-\epsilon(r))/2}\sigma_{ss}\bar\sigma_{s,-t}\bar a_1)
\end{align*}
and
\begin{align*}
\tau_2:=&T_{rt}(\sigma_{rs}\bar\sigma_{rr}a_1)T_{st}(-\sigma_{rr}\bar\sigma_{rr}a_1)T_{r,-s}(-\lambda^{(\epsilon(t)-\epsilon(s))/2}\sigma_{r,-t}\bar\sigma_{rr}\bar a_1)\cdot\\&\cdot T_{r}(0,\lambda^{(\epsilon(t)-\epsilon(r))/2}\sigma_{r,-t}\bar\sigma_{rs}\bar a_1-\lambda^{(-\epsilon(t)-\epsilon(r))/2}\sigma_{rs}\bar\sigma_{r,-t}a_1).
\end{align*}
Clearly $\tau_1,\tau_2\in \EU_{7}(I,\Omega)$. Set $\xi_1:=\sigma\tau_1^{-1}\sigma^{-1}$ and $\xi_2:=\sigma\tau_2^{-1}\sigma^{-1}$. One checks easily that $(\sigma\tau_1^{-1})_{s*}=\sigma_{s*}$ and $(\tau_1^{-1}\sigma^{-1})_{*,-s}=\sigma'_{*,-s}$. Hence $(\xi_1)_{s*}=e_{s*}$ and $(\xi_1)_{*,-s}=e_{*,-s}$. Similarly $(\sigma\tau_2^{-1})_{r*}=\sigma_{r*}$ and $(\tau_2^{-1}\sigma^{-1})_{*,-r}=\sigma'_{*,-r}$. Hence $(\xi_2)_{r*}=e_{r*}$ and $(\xi_2)_{*,-r}=e_{*,-r}$. A straightforward computation shows that 
\[(\tau_1,\xi_1)\xrightarrow{T_{-s,t}(a_2),T_{-s,r}(a_3),T_{t,\pm r}(-\lambda^{(\epsilon(s)-\epsilon(t))/2}a_4)}(T_{-s,\pm r}(\sigma_{s,-t}\bar\sigma_{sr}a_1a_2a_3a_4),e)\]
and
\[(\tau_2,\xi_2)\xrightarrow{T_{sr}(a_2),T_{tr}(a_3),T_{r,-t}(a_4),T_{-r,s}(a_5)}(T_{-r,-t}(\sigma_{rs}\bar\sigma_{rr}a_1a_2a_3a_4a_5),e).\]
It follows from Lemma \ref{lemredux} that $H$ contains the matrices $T_{-s,\pm r}(\sigma_{s,-t}\bar\sigma_{sr}a_1a_2a_3a_4)$ and $T_{-r,-t}(\sigma_{rs}\bar\sigma_{rr}a_1a_2a_3a_4a_5)$. Clearly $H$ also contains $T_{-r,t}(\sigma_{rs}\bar\sigma_{rr}a_1a_2a_3a_4a_5)$ (just replace $t$ by $-t$ in the argument above). The assertion of the lemma follows from Lemma \ref{lemsub1}. 
\end{proof}

\begin{lemma}\label{lemsub3}
Let $(I,\Omega)$ be an odd form ideal and $H$ a subgroup of $\U_{7}(R,\Delta)$ normalised by $\EU_{7}(I,\Omega)$. Then $T_{ij}(\sigma_{rs}a)\in H$ for any $i,j\in \Theta_{\hb}, i\neq \pm j$, $\sigma\in H$, $r,s
\in\Theta_{\hb},r\neq \pm s$ and $a\in I^9$.
\end{lemma}
\begin{proof}
Choose a $t\in \Theta_{hb}$ such that $t\neq\pm r,\pm s$. Let $a_0\in I$ and set 
\[\tau:=[\sigma^{-1},T_{tr}(-\bar\sigma_{rs}\bar a_0)]\in H.\]
Let $J$ be the involution invariant ideal generated by the set $\{a_0\sigma_{rs}\bar\sigma_{rr},a_0\sigma_{rs}\bar\sigma_{r,\pm t}\}$.
Clearly
\begin{align*}
\tau_{tt}&=1-\sigma'_{tt}\bar\sigma_{rs}\bar a_0\sigma_{rt}+\sigma_{t,-r}'\lambda^{(\epsilon(r)-\epsilon(t))/2}\sigma_{rs}a_0\sigma_{-t,t}\\
&=1-\sigma'_{tt}\underbrace{\bar\sigma_{rs}\bar a_0\sigma_{rt}}_{\in J}+\lambda^{-\epsilon(t)}\underbrace{\bar\sigma_{r,-t}\sigma_{rs} a_0}_{\in J}\sigma_{-t,t}\\
&\equiv 1\bmod J
\end{align*}
and 
\begin{align*}
\tau_{tr}&=-\sigma'_{tt}\bar\sigma_{rs}\bar a_0\sigma_{rr}+\sigma_{t,-r}'\lambda^{(\epsilon(r)-\epsilon(t))/2}\sigma_{rs}a_0\sigma_{-t,r}+\tau_{tt}\bar\sigma_{rs}\bar a_0\\
&=-\sigma'_{tt}\underbrace{\bar\sigma_{rs}\bar a_0\sigma_{rr}}_{\in J}+\lambda^{-\epsilon(t)}\underbrace{\bar\sigma_{r,-t}\sigma_{rs} a_0}_{\in J}\sigma_{-t,r}+\tau_{tt}\bar\sigma_{rs}\bar a_0\\
&\equiv \bar\sigma_{rs}\bar a_0\bmod J
\end{align*}
by Lemma \ref{36}. Hence $\tau_{tt}\bar\tau_{tr}\equiv \sigma_{rs}a_0\bmod J$. By Lemma \ref{lemsub2}(ii) we have $T_{ij}(\tau_{tr}\bar\tau_{tt}a)\in H$ for any $i,j\in \Theta_{\hb}, i\neq \pm j, a\in I^8$ whence $T_{ij}(\tau_{tt}\bar\tau_{tr}a)\in H$ for any $i,j\in \Theta_{\hb}, i\neq \pm j, a\in I^8$ (because of Relation (S1) in Lemma \ref{39}). Since $\tau_{tt}\bar\tau_{tr}a= \sigma_{rs}a_0a+xa$ for some $x\in J$, we obtain the assertion of the lemma in view of Lemma \ref{lemsub2}.
\end{proof}

\begin{lemma}\label{lemsub4}
Let $(I,\Omega)$ be an odd form ideal and $H$ a subgroup of $\U_{7}(R,\Delta)$ normalised by $\EU_{7}(I,\Omega)$. Let $(J,\Sigma)=U(H)$. Then 
$\EU_{7}(J I^{12},\Omega_{\min}^{J I^{12}})\leq H$.
\end{lemma}
\begin{proof}
Let $\sigma\in \U_{2n+1}(R,\Delta)$ and $i,j,r,s\in\Theta_{\hb}$ such that $i\neq\pm j$ and $r\neq \pm s$. Furthermore, let $x,y\in J(\Delta)$. Suppose we have shown
\begin{enumerate}[(i)]
\item $T_{ij}(\sigma_{rs}a)\in H$ for any $a\in I^{9}$,
\item $T_{ij}(\sigma_{r,-r}a)\in H$ for any $a\in I^{10}$,
\item $T_{ij}(\sigma_{r0}xa)\in H$ for any $a\in I^{10}$,
\item $T_{ij}(\bar x \mu\sigma_{0s}a)\in H$ for any $a\in I^{10}$,
\item $T_{ij}((\sigma_{rr}-\sigma_{ss})a)\in H$ for any $a\in I^{10}$,
\item $T_{ij}((\sigma_{rr}-\sigma_{-r,-r})a)\in H$ for any $a\in I^{10}$ and
\item $T_{ij}(\bar x\mu (\sigma_{00}-\sigma_{ss})ya)\in H$ for any $a\in I^{11}$.
\end{enumerate}
Then it would follow that $\EU_{7}(J I^{12},\Omega_{\min}^{J I^{12}})\leq H$ in view of Relations (S5) and (SE2) in Lemma \ref{39}. Hence it suffices to show (i)-(vii) above.
\begin{enumerate}[(i)]
\item Follows from the previous lemma.
\item Let $b\in I$ and $c\in I^9$. Clearly the entry of $^{T_{sr}(b)}\sigma\in H$ at position $(s,-r)$ equals $\sigma_{r,-r}b+\sigma_{s,-r}$. It follows from the previous lemma that $T_{ij}((\sigma_{r,-r}b+\sigma_{s,-r})c)\in H$. But $T_{ij}(\sigma_{s,-r}c)\in H$, again by the previous lemma. Hence $T_{ij}(\sigma_{r,-r}bc)\in H$.
\item Let $b\in I$ and $c\in I^9$. Choose a $z$ such that $(x,z)\in \Delta^{\epsilon(s)}$ (possible since $x\in J(\Delta)$). Clearly the entry of $\sigma^{T_{-s}(xb,\bar bzb)}\in H$ at position $(r,s)$ equals $\sigma_{r0}xb+\sigma_{rs}+\sigma_{r,-s}\bar bzb$. It follows from the previous lemma that $T_{ij}((\sigma_{r0}xb+\sigma_{rs}+\sigma_{r,-s}\bar bzb)c)\in H$. But $T_{ij}(\sigma_{rs}c),T_{rs}(\sigma_{r,-s}\bar bzbc)\in H$, again by the previous lemma. It follows that $T_{ij}(\sigma_{r0}xbc)\in H$.
\item Let $b\in I$ and $c\in I^9$. Choose a $z$ such that $(x,z)\in \Delta^{\epsilon(-r)}$. Clearly the entry of $^{T_{r}(x\bar b,bz\bar b)}\sigma\in H$ at position $(r,s)$ equals $-\lambda^{-(1+\epsilon(r))/2}b\bar x\mu\sigma_{0s}+\sigma_{rs}+bz\bar b\sigma_{-r,s}$. It follows from the previous lemma that $T_{ij}((-\lambda^{-(1+\epsilon(r))/2} b\bar x\mu\sigma_{0s}+\sigma_{rs}+bz\bar b\sigma_{-r,s})c)\in H$. But $T_{ij}(\sigma_{rs}c),T_{rs}(bz\bar b\sigma_{-r,s}c)\in H$, again by the previous lemma. It follows that $T_{ij}(-\lambda^{-(1+\epsilon(r))/2}b\bar x\mu\sigma_{0s}c)\in H$.
\item Let $b\in I$ and $c\in I^9$. One checks easily that the entry of $^{T_{sr}(b)}\sigma\in H$ at position $(s,r)$ equals $(\sigma_{rr}-\sigma_{ss})b+\sigma_{sr}-\sigma_{rs}b^2$. It follows from the previous lemma that $T_{ij}(((\sigma_{rr}-\sigma_{ss})b+\sigma_{sr}-\sigma_{rs}b^2)c)\in H$. But $T_{ij}(\sigma_{sr}c),T_{rs}(-\sigma_{rs}b^2c)\in H$, again by the previous lemma. It follows that $T_{ij}((\sigma_{rr}-\sigma_{ss})bc)\in H$.
\item Follows from (v) since $T_{ij}((\sigma_{rr}-\sigma_{-r,-r})a)=T_{ij}((\sigma_{rr}-\sigma_{ss})a)T_{ij}((\sigma_{ss}-\sigma_{-r,-r})a)$.
\item Let $b\in I$ and $c\in I^{10}$. Choose a $z$ such that $(y,z)\in \Delta^{\epsilon(s)}$. One checks easily that the entry of $\sigma{}^{T_{-s}(yb,\bar bzb)}\in H$ at position $(0,s)$ equals $(\sigma_{00}-\sigma_{ss})yb+\sigma_{0s}-\sigma_{s0}y^2b^2+\sigma_{0,-s}\bar bzb-\sigma_{s,-s}yb\bar bzb$. It follows from (iv) that 
\[T_{ij}(\bar x\mu((\sigma_{00}-\sigma_{ss})yb+\sigma_{0s}-\sigma_{s0}y^2b^2+\sigma_{0,-s}\bar bzb-\sigma_{s,-s}yb\bar bzb)c)\in H.\]
But $T_{ij}(\bar x\mu\sigma_{0s}c),T_{ij}(-\bar x\mu\sigma_{s0}y^2b^2c),T_{ij}(\bar x\mu\sigma_{0,-s}\bar bzbc),T_{ij}(-\bar x\mu\sigma_{s,-s}yb\bar bzbc)\in H$ by (ii),(iii) and (iv). It follows that $T_{ij}(\bar x\mu(\sigma_{00}-\sigma_{ss})ybc)\in H$.
\end{enumerate}
\end{proof}

We introduce the following notation. If $(x,y)\in \Delta$, we set $(x,y)^{1}:=(x,y)$ and $(x,y)^{-1}:=(x,\bar\lambda y)$. Note that $\Delta$ contains the element $\minus(x,y)\circ(-1)=(-x,\lambda \bar y)\circ(-1)=(x,\lambda \bar y)$. Hence $(x,y)^{-1}\in \Delta^{-1}$. On the other hand any element of $\Delta^{-1}$ is equal to $(x,y)^{-1}$ for some $(x,y)\in\Delta$. Hence $\Delta^{-1}=\{(x,y)^{-1}\mid (x,y)\in \Delta\}$. Similarly, if $\Omega$ is a relative odd form parameter, then $\Omega^{-1}=\{(x,y)^{-1}\mid (x,y)\in \Omega\}$. One checks easily that $(\alpha\+\beta)^{-1}=\alpha^{-1}\+_{-1}\beta^{-1}$ and $(\alpha\circ x)^{-1}=\alpha^{-1}\circ x$ for any $\alpha,\beta\in \Delta$ and $x\in R$.

In the proof of Theorem \ref{thmm1} we will use the matrices $T_{*j}(u,x)$ defined below. These matrices are examples of ESD transvections, cf. \cite{petrov}.

\begin{definition}\label{defesd}
Let $(I,\Omega)$ be an odd form ideal and $j\in \Theta_{hb}$. Moreover, let $u\in M$ and $x\in R$ such that $u_i\in I$ for any $i\in\Theta_{hb}$, $u_{j}=0$ and $Q(u)^{\epsilon(j)}\+(0,x)\in\Omega^{\epsilon(j)}$. We define the matrix 
\begin{align*}
T_{*j}(u,x):=&e+ue^t_{j}- e_{-j}\lambda^{(\epsilon(j)-1)/2}\tilde u+xe^{-j,j}\\
=&(\prod\limits_{i\in \Theta_{hb}\setminus\{\pm j\}}T_{ij}(u_i))T_{-j}(Q(u)^{\epsilon(j)}\+(0,x)\+(0,u_{-j}-\lambda^{\epsilon(j)}\bar u_{-j}))\\
\in& \EU_{7}(I,\Omega).
\end{align*}
Instead of $T_{*j}(u,0)$ we may write $T_{*j}(u)$. Clearly $T_{*j}(u)^{-1}=T_{*j}(-u)$ (note that $\tilde uu=tr(Q(u))=0$ since $Q(u)\in\Omega\subseteq \Delta_{\max}$) and
\begin{align}
^{\sigma}T_{*j}(u)&=e+\sigma ue^t_{j}\sigma^{-1}- \sigma e_{-j} \lambda^{(\epsilon(j)-1)/2}\tilde u\sigma^{-1}\notag\\
&=e+ \lambda^{(-\epsilon(j)-1)/2}\sigma u\widetilde{\sigma_{*,-j}}- \lambda^{(\epsilon(j)-1)/2}\sigma_{*,-j}\widetilde{\sigma u}\label{e3}
\end{align} 
for any $\sigma\in \U_{7}(R,\Delta)$, the last equality by Lemma \ref{38}.
\end{definition}

\begin{theorem}\label{thmm1}
Let $(I,\Omega)$ be an odd form ideal and $H$ a subgroup of $\U_{7}(R,\Delta)$ normalised by $\EU_{7}(I,\Omega)$. Then 
\[\EU_{7}(U(H)\ast I^{12})\leq H\leq \CU_{7}((R,\Delta),U(H)).\]
\end{theorem}
\begin{proof}
The inclusion $H\subseteq \CU_{7}((R,\Delta),U(H))$ holds by Proposition \ref{propul}. It remains to show the inclusion $\EU_{7}(U(H)\ast I^{12})\subseteq H$. Recall that if $U(H)=(J,\Sigma)$, then $U(H)\ast I^{12}=(JI^{12}, \Omega_{\min}^{JI^{12}}\+\Sigma\circ I^{12})$. Because of Lemma \ref{lemsub4} it suffices to show that $T_i(\alpha)\in H$ for any $i\in \Theta_{\hb}$ and $\alpha \in (\Sigma\circ I^{12})^{-\epsilon(i)}$, or equivalently 
\begin{equation}
T_i(\alpha^{-\epsilon(i)})\in H\text{ for any }i\in \Theta_{\hb}\text{ and }\alpha \in \Sigma\circ I^{12}.
\end{equation}
Recall that $\Sigma=\Omega^J_{\min}\+Z\circ R$ where 
\[Z=\{Q(\sigma_{*s}),(Q(\sigma_{*0})\minus(1,0))\circ y\+(y,z)\minus (y,z)\circ \sigma_{ss}\mid \sigma\in H, s\in \Theta_{\hb}, (y,z)\in \Delta
\}.\]
In order to prove (5) it suffices to show (i) and (ii) below (note that $\Omega^J_{\min}\circ I^{12}\subseteq \Omega^{JI^{12}}_{\min}$).
\begin{enumerate}[(i)]
\item $T_{i}((Q(\sigma_{*s})\circ a)^{-\epsilon(i)})\in H$ for any $\sigma\in H$, $i,s\in\Theta_{\hb}$ and $a\in I^{12}$.
\item $T_{i}((((Q(\sigma_{*0})\minus (1,0))\circ y \+ (y,z)\minus (y,z)\circ \sigma_{ss})\circ a)^{-\epsilon(i)})\in H$ for any $\sigma\in H$, $i,s\in\Theta_{\hb}$, $(y,z)\in \Delta$ and $a\in I^{12}$.
\end{enumerate}
We first show (i) and then (ii).
\begin{enumerate}[(i)]
\item Let $\sigma\in H$. In Step 1 below we show that $T_{i}((Q(\sigma_{*s})\circ \sigma_{ss}a)^{-\epsilon(i)})\in H$ for any $i,s\in \Theta_{\hb}$, $i\neq \pm s$ and $a\in I^{10}$. In Step 2 we use Step 1 to show that $T_{i}((Q(\sigma_{*s})\circ a)^{-\epsilon(i)})\in H$ for any $i,s\in \Theta_{\hb}$, $i\neq \pm s$ and $a\in I^{11}$. In Step 3 we use Step 2 to prove (i).\\
\\
{\bf Step 1.} Let $r,s\in \Theta_{+}, r\neq s$ and $b\in I^9$. Set 
\[u':=e_{-r}\sigma'_{-s,-s}-e_{-s}\sigma'_{-s,-r}=e_{-r}\bar\sigma_{ss}-e_{-s}\bar\sigma_{rs}\in M
\]
and $u:=\sigma^{-1}u'b\in M$. Then clearly $u_{-s}=0$. Moreover, $u_i\in I$ for any $i\in\Theta_{hb}$ and $Q(u)=Q(\sigma^{-1}u'b)=Q(\sigma^{-1}u')\circ b\in\Omega^I_{\min}$ since $Q(u')=0$ and $\sigma^{-1}$ preserves $Q$ modulo $\Delta$. Hence the matrices $T_{*,-s}(u)$ and $T_{*,-s}(-u)$ are defined and are contained in $\EU_{7}(I,\Omega)$ (see Definition \ref{defesd}). Set
\[
\xi:={}^{\sigma}T_{*,-s}(-u)\overset{(4)}{=}e-\sigma u\widetilde{\sigma_{*s}}+ \sigma_{*s} \bar\lambda\widetilde{\sigma u}=e-u'b\widetilde{\sigma_{*s}}+ \sigma_{*s}\bar\lambda \widetilde{u'b}.
\]
Choose a $t\in\Theta_{hb}$ such that $t\neq \pm r,\pm s$. Set
\[\tau:=T_{tr}(-\sigma_{ts}\sigma_{ss}\bar b)T_{ts}(\sigma_{ts}\sigma_{rs}\bar b)\]
and $\zeta:=\xi\tau$. Note that by Lemma \ref{lemsub3} we have $\tau\in H$. With a little effort one can check that $\zeta_{t*}=e_{t*}$, $\zeta_{*,-t}=e_{*,-t}$
and
\begin{align*}
\zeta_{*r}=&e_{r}+(\sigma_{*s}-e_t\sigma_{ts})\sigma_{ss}\bar b+e_{-s}(\bar\sigma_{rs}\bar\sigma_{-r,s}b\lambda-\sigma_{ts}\sigma_{ss}\bar b\bar \sigma_{rs}\bar \sigma_{-t,s}b\lambda^{(\epsilon(t)+1)/2})\\
&+e_{-r}(-\bar\sigma_{ss}\bar\sigma_{-r,s}b\lambda+\sigma_{ts}\sigma_{ss}\bar b\bar \sigma_{ss}\bar \sigma_{-t,s}b\lambda^{(\epsilon(t)+1)/2}).
\end{align*}
Let $c\in I$. Clearly 
\[(T_{*,-s}(u),\zeta)\xrightarrow{T_{rt}(-c)}(\phi,\psi)\]
where 
\begin{align*}
\phi=&[T_{*,-s}(-u),T_{rt}(-c)]\\
=&T_{s}(0,-
u_{t}\bar u_{-r}c+\overline{u_{t}\bar u_{-r}c}\bar\lambda)T_{s,-r}(\bar\lambda\bar u_{t}\bar c )T_{st}(-\bar u_{-r}c)
\end{align*}
and 
\begin{align*}
\psi=&[T_{rt}(-c),\zeta]\\
=&T_{rt}(-c)T_{*t}(\zeta_{*r}c)\\
=&T_{*t}((\zeta_{*r}-e_{*r})c,\lambda^{(\epsilon(t)-1)/2}\bar c\zeta_{-r,r}c)\\
=&(\prod\limits_{i\in \Theta_{hb}\setminus\{\pm t\}}T_{it}((\zeta_{ir}-\delta_{ir})c))T_{-t}(0,\zeta_{-t,r}c-\lambda^{\epsilon(t)}\overline \zeta_{-t,r}c)\cdot\\
&\cdot T_{-t}(Q((\zeta_{*r}-e_{*r})c)^{\epsilon(t)}\+(0,\lambda^{(\epsilon(t)-1)/2}\bar c\zeta_{-r,r}c))).
\end{align*}
It follows from Lemma \ref{lemredux} that $\phi\psi\in H$ (since $T_{*,-s}(u)\zeta= [T_{*,-s}(u),\sigma]\tau\in H$). Clearly
\begin{align}
\phi\psi=&T_{s}(0,-
u_{t}\bar u_{-r}c+\overline{u_{t}\bar u_{-r}c}\bar\lambda)T_{s,-r}(\bar\lambda\bar u_{t}\bar c )T_{st}((\zeta_{sr}-\bar u_{-r})c)\cdot\notag\\
&\cdot(\prod\limits_{i\in \Theta_{hb}\setminus\{s,\pm t\}}T_{it}((\zeta_{ir}-\delta_{ir})c))T_{-t}(0,\zeta_{-t,r}c-\lambda^{\epsilon(t)}\overline \zeta_{-t,r}c)\cdot\notag\\
&\cdot T_{-t}(Q((\zeta_{*r}-e_{*r})c)^{\epsilon(t)}\+(0,\lambda^{(\epsilon(t)-1)/2}\bar c\zeta_{-r,r}c))).
\end{align}
We leave it to the reader to deduce from Lemma \ref{lemsub3} that all the factors on the right hand side of Equation (6) except the last one are contained in $H$ (cf. the proof of Lemma \ref{lemsub4}). Since $\phi\psi\in H$, it follows that 
\[T_{-t}(Q((\zeta_{*r}-e_{*r})c)^{\epsilon(t)}\+(0,\lambda^{(\epsilon(t)-1)/2}\bar c\zeta_{-r,r}c))\in H.\]
But one checks easily that
\[Q((\zeta_{*r}-e_{*r})c)^{\epsilon(t)}\+(0,\lambda^{(\epsilon(t)-1)/2}\bar c\zeta_{-r,r}c)=(Q(\sigma_{*s})\circ\sigma_{ss}\bar bc)^{\epsilon(t)}\+(0,y-\bar y\lambda^{\epsilon(t)})\]
for some $y$ that lies in the ideal generated by $\sigma_{-r,s}\bar bc$ and $\sigma_{ts}\bar b c$. It follows that $T_{-t}((Q(\sigma_{*s})\circ\sigma_{ss}\bar bc)^{\epsilon(t)})\in H$. Thus we have shown that 
\[T_{-t}((Q(\sigma_{*s})\circ\sigma_{ss}a)^{\epsilon(t)})\in H
\]
for any $s\in\Theta_+$, $t\in \Theta_{hb},t\neq\pm s$ and $a\in I^{10}$. Analogously one can show that 
\[T_{-t}((Q(\sigma_{*s})\circ\sigma_{ss}a)^{\epsilon(t)})\in H
\]
for any $s\in\Theta_-$, $t\in \Theta_{hb},t\neq\pm s$ and $a\in I^{10}$. \\
\\
{\bf Step 2.} Let $i,s\in\Theta_{\hb},i\neq \pm s$ and $a\in I^{11}$. By Lemma \ref{last} we have
\begin{align*}
&T_{i}((Q(\sigma_{*s})\circ a)^{-\epsilon(i)})\\
=&T_{i}(Q(\sigma_{*s})^{-\epsilon(i)}\circ a)\\
=&T_{i}(Q(\sigma_{*s})^{-\epsilon(i)}\circ\sum\limits_{p=1}^{-1}\sigma'_{sp}\sigma_{ps}a)\\
=&\underbrace{\prod\limits_{p=1}^{-1}T_i(Q(\sigma_{*s})^{-\epsilon(i)}\circ\sigma'_{sp}\sigma_{ps}a)}_{X:=}\cdot\\
&\cdot\underbrace{\prod\limits_{\substack{p,q=1,\\p\succ q}}^{-1}T_i(0, \overline{\sigma'_{sp}\sigma_{ps}}aQ_2(\sigma_{*s})\sigma'_{sq}\sigma_{qs}a-\overline{\overline{\sigma'_{sp}\sigma_{ps}a}Q_2(\sigma_{*s})\sigma'_{sq}\sigma_{qs}a}\lambda^{-\epsilon(i)})}_{Y:=}.
\end{align*}
By Step 1, $T_{i}(Q(\sigma_{*s})^{-\epsilon(i)}\circ \sigma'_{ss}\sigma_{ss}a)\in H$. All the other factors of $X$ are also contained in $H$ (see the proof of Lemma \ref{lemsub4}). Similarly $Y$ is contained in $H$. It follows that $T_{i}((Q(\sigma_{*s})\circ a)^{-\epsilon(i)})\in H$. \\
\\
{\bf Step 3} By Step 2 we know that $T_{i}((Q(\sigma_{*s})\circ a)^{-\epsilon(i)})\in H$ for any $i,s\in \Theta_{\hb}$, $i\neq \pm s$ and $a\in I^{11}$. In order to show (i) it remains to show that 
\begin{equation}
T_{i}((Q(\sigma_{*s})\circ a)^{-\epsilon(i)})\in H\text{ for any }s\in \Theta_{\hb}, i\in\{ \pm s\}\text{ and }a\in I^{12}.
\end{equation}
Let $s,r\in\Theta_{hb},s\neq \pm r$, $b\in I$ and set $\tau:=[\sigma,T_{sr}(b)]\in H$. Applying Step 2 to $\tau$ we obtain $T_{i}((Q(\tau_{*r})\circ c)^{-\epsilon(i)})\in H$ for any $i\in\{\pm s\}$ and $c\in I^{11}$. We leave it to the reader to deduce from \cite[Lemma 63]{bak-preusser} that $T_{i}((Q(\sigma_{*s})\circ bc\sigma'_{rr})^{-\epsilon(i)})\in H$ for any $i\in\{\pm s\}$ and $c\in I^{11}$. It follows as in Step 2 that $T_{i}((Q(\sigma_{*s})\circ bc)^{-\epsilon(i)})\in H$ for any $i\in\{\pm s\}$ and $c\in I^{11}$. Thus (7) holds.

\item Let $\sigma\in H$, $s\in\Theta_{\hb}$, $(y,z)\in \Delta$ and $b\in I$. Set $\rho:=[\sigma,T_{s}(((y,z)\circ b)^{-\epsilon(s)})]\in H$. By Step 2 in the proof of (i) above, we have $T_{i}((Q(\rho_{*,-s})\circ c)^{-\epsilon(i)})\in H$ for any $i\neq \pm s$ and $c\in I^{11}$. We leave it to the reader to deduce from \cite[Lemma 63]{bak-preusser} that 
\[T_{i}(((Q(\sigma_{*0})\minus(1,0))\circ y\sigma'_{-s,-s}\+(y,z)\circ \sigma'_{-s,-s}\minus(y,z))\circ bc)^{-\epsilon(i)})\in H\]
for any $i\neq \pm s$ and $c\in I^{11}$. It follows as in Step 2 that
\begin{equation}
T_{i}(((Q(\sigma_{*0})\minus(1,0))\circ y\+(y,z) \minus(y,z)\circ\sigma_{-s,-s})\circ bc)^{-\epsilon(i)})\in H
\end{equation}
for any $i\neq \pm s$ and $c\in I^{11}$ (cf. the proof of Lemma \ref{lemul2}). Let $r\in \Theta_{\hb}$. One checks easily that 
\begin{equation}
(y,z)\circ\sigma_{-s,-s}=(y,z)\circ \sigma_{rr}\+ (y,z)\circ (\sigma_{-s,-s}-\sigma_{rr})\+\alpha
\end{equation}
where $\alpha=(0,\overline{\sigma_{-s,-s}-\sigma_{rr}}z\sigma_{rr}-\overline{\overline{\sigma_{-s,-s}-\sigma_{rr}}z\sigma_{rr}}\lambda)$. Clearly 
\begin{equation}
T_{i}((((y,z)\circ (\sigma_{-s,-s}-\sigma_{rr})\+\alpha)\circ bc)^{-\epsilon(i)})\in H.
\end{equation}
It follows from (8), (9) and (10) that
\begin{equation}
T_{i}(((Q(\sigma_{*0})\minus(1,0))\circ y\+(y,z) \minus(y,z)\circ\sigma_{rr})\circ bc)^{-\epsilon(i)})\in H
\end{equation}
for any $i\neq \pm s$ and $c\in I^{11}$. Since (11) holds for any $s,r\in \Theta_{hb}$, we have shown (ii).
\end{enumerate}
\end{proof}

\begin{theorem}\label{thmm2}
Let $(I,\Omega)$ be an odd form ideal and $H$ a subgroup of $\U_{7}(R,\Delta)$ normalised by $\EU_{7}((R,\Delta),(I,\Omega))$. Then 
\[\EU_{7}((R,\Delta),U(H)\ast I^{12})\leq H\leq \CU_{7}((R,\Delta),U(H)).\]
\end{theorem}
\begin{proof}
By Proposition \ref{propul} we only have to show that $\EU_{7}((R,\Delta),U(H)\ast I^{12})\leq H$. Let $\tau\in \EU_{7}(R,\Delta)$. Then clearly $H^\tau$ is normalised by $\EU_{7}((R,\Delta),(I,\Omega))$. By Corollary \ref{corul2} we have $U(H^{\tau})=U(H)$. It follows from the previous theorem that $\EU_{7}(U(H)\ast I^{12})\leq H^\tau$ whence $\EU_{7}(U(H)\ast I^{12})^{\tau^{-1}}\leq H$. Since this holds for any $\tau\in \EU_{7}(R,\Delta)$, we obtain $\EU_{7}((R,\Delta),(U(H)\ast I^{12}))\leq H$.
\end{proof}

\begin{theorem}\label{thmm3}
Suppose that $H\lhd_d G$ where $d$ is a positive integer and $G$ a subgroup of $\U_{7}(R,\Delta)$ containing $\EU_{7}(R,\Delta)$. Suppose $U(H)=(I,\Omega)$. Then \[\EU_{7}((R,\Delta),U(H)\ast I^{k})\leq H \leq \CU_{7}((R,\Delta),U(H))\]
where $k=\frac{12^d-1}{11}-1$.
\end{theorem}
\begin{proof}
By Proposition \ref{propul} we only have to show that $\EU_{7}((R,\Delta),U(H)\ast I^{k})\leq H$. We proceed by induction on $d$.\\
\\
\underline{$d=1$:} If $d=1$, then $H\lhd G$ and therefore $H$ is normalised by $\EU_{7}(R,\Delta)$. It follows from the previous theorem that $\EU_{7}((R,\Delta),U(H))\leq H$ as desired (we use the convention $I^0=R$).\\
\\
\underline{$d\rightarrow d+1$:} Suppose $H\lhd_{d+1}G$, i.e. $H=H_0\lhd H_1\lhd \dots\lhd H_{d}\lhd H_{d+1}=G$ for some subgroups $H_1,\dots,H_d$ of $G$. Write $U(H)=(I,\Omega)$ and $U(H_1)=(J,\Sigma)$. By the induction assumption we have 
\begin{equation}
\EU_{7}((R,\Delta),U(H_1)\ast J^{k})\leq H_1 \leq \CU_{7}((R,\Delta),U(H_1))
\end{equation}
where $k=\frac{12^d-1}{11}-1$. Since $H\lhd H_1$, it follows that $H$ is normalised by $\EU_{7}((R,\Delta),$ $U(H_1)\ast J^{k})$. Hence
\begin{equation*}
\EU_{7}((R,\Delta),U(H)\ast J^{12(k+1)})\leq H,
\end{equation*}
by the previous theorem. It follows from (12) that $H\leq \CU_{7}((R,\Delta),$ $U(H_1))$ whence $U(H)\subseteq U(H_1)$. Therefore $I\subseteq J$ and thus we obtain \begin{equation*}
\EU_{7}((R,\Delta),U(H)\ast I^{12(k+1)})\leq H.
\end{equation*}
One checks easily that $12(k+1)=\frac{12^{d+1}-1}{11}-1$ as desired.
\end{proof}

\section{The subnormal structure of the groups $\U_{2n+1}(R,\Delta)$ where $n\geq 4$}

In this section $n$ denotes an integer greater than or equal to $4$ and $(R,\Delta)$ a Hermitian form ring where $R$ is commutative. We start by proving an analogue of Lemma \ref{lemsub2}.

\begin{lemma}\label{4lemsub2}
Let $(I,\Omega)$ be an odd form ideal and $H$ a subgroup of $\U_{2n+1}(R,\Delta)$ normalised by $\EU_{2n+1}(I,\Omega)$. Let $\sigma\in H$ and $r,s,t
\in\Theta_{\hb}$ such that $r\neq \pm s$ and $t\neq\pm r, \pm s$. Then 
\begin{enumerate}[(i)]
\item $T_{ij}(\sigma_{s,-t}\bar\sigma_{sr}a)\in H$ for any $i,j\in \Theta_{\hb}, i\neq \pm j, a\in I^6$ and
\item $T_{ij}(\sigma_{rs}\bar\sigma_{rr}a)\in H$ for any $i,j\in \Theta_{\hb}, i\neq \pm j, a\in I^6$.
\end{enumerate}
\end{lemma}
\begin{proof}
Let $a_1,\dots,a_4
\in I$. Set 
\begin{align*}
\tau_1:=&T_{rt}(\sigma_{ss}\bar\sigma_{sr}\bar a_1)T_{st}(-\sigma_{sr}\bar\sigma_{sr}\bar a_1)T_{r,-s}(-\lambda^{(\epsilon(t)-\epsilon(s))/2}\sigma_{s,-t}\bar\sigma_{sr}a_1)\cdot\\&\cdot T_{r}(0,\lambda^{(\epsilon(t)-\epsilon(r))/2}\sigma_{s,-t}\bar\sigma_{ss} a_1-\lambda^{(-\epsilon(t)-\epsilon(r))/2}\sigma_{ss}\bar\sigma_{s,-t}\bar a_1)
\end{align*}
and
\begin{align*}
\tau_2:=&T_{rt}(\sigma_{rs}\bar\sigma_{rr}a_1)T_{st}(-\sigma_{rr}\bar\sigma_{rr}a_1)T_{r,-s}(-\lambda^{(\epsilon(t)-\epsilon(s))/2}\sigma_{r,-t}\bar\sigma_{rr}\bar a_1)\cdot\\&\cdot T_{r}(0,\lambda^{(\epsilon(t)-\epsilon(r))/2}\sigma_{r,-t}\bar\sigma_{rs}\bar a_1-\lambda^{(-\epsilon(t)-\epsilon(r))/2}\sigma_{rs}\bar\sigma_{r,-t}a_1).
\end{align*}
Clearly $\tau_1,\tau_2\in \EU_{2n+1}(I,\Omega)$. Set $\xi_1:=\sigma\tau_1^{-1}\sigma^{-1}$ and $\xi_2:=\sigma\tau_2^{-1}\sigma^{-1}$. One checks easily that $(\sigma\tau_1^{-1})_{s*}=\sigma_{s*}$ and $(\tau_1^{-1}\sigma^{-1})_{*,-s}=\sigma'_{*,-s}$. Hence $(\xi_1)_{s*}=e_{s*}$ and $(\xi_1)_{*,-s}=e_{*,-s}$. Similarly $(\sigma\tau_2^{-1})_{r*}=\sigma_{r*}$ and $(\tau_2^{-1}\sigma^{-1})_{*,-r}=\sigma'_{*,-r}$. Hence $(\xi_2)_{r*}=e_{r*}$ and $(\xi_2)_{*,-r}=e_{*,-r}$. A straightforward computation shows that 
\[(\tau_1,\xi_1)\xrightarrow{T_{-s,t}(a_2),T_{-s,r}(a_3),T_{tu}(-\lambda^{(\epsilon(s)-\epsilon(t))/2}a_4)}(T_{-s,u}(\sigma_{s,-t}\bar\sigma_{sr}a_1a_2a_3a_4),e)\]
for any $u\neq \pm s,\pm t$, and
\[(\tau_1,\xi_1)\xrightarrow{T_{-s,u}(a_2),T_{-s,r}(a_3),T_{uv}(-\lambda^{(\epsilon(s)-\epsilon(t))/2}a_4)}(T_{-s,v}(\sigma_{s,-t}\bar\sigma_{sr}a_1a_2a_3a_4),e)\]
for any $u\neq \pm r, \pm s,\pm t$ and $v\neq \pm s,\pm u$. It follows from Lemma \ref{lemredux} that $H$ contains the matrices $T_{-s,u}(\sigma_{s,-t}\bar\sigma_{sr}a_1a_2a_3a_4)$ where $u\neq \pm s$. It is now an easy exercise to obtain (i) (see the proof of Lemma \ref{lemsub1}).\\
On the other hand one checks easily that 
\[(\tau_2,\xi_2)\xrightarrow{T_{sr}(a_2),T_{tu}(a_3),T_{vs}(-a_4)}(T_{vu}(\sigma_{rs}\bar\sigma_{rr}a_1a_2a_3a_4),e)\]
for any $u\neq \pm r, \pm s,\pm t$ and $v\neq r,\pm s,-t,\pm u$. It follows from Lemma \ref{lemredux} that 
\begin{equation}
T_{vu}(\sigma_{rs}\bar\sigma_{rr}a_1a_2a_3a_4)\in H\text{ for any }u\neq \pm r, \pm s,\pm t \text{ and }v\neq r,\pm s,-t,\pm u.
\end{equation}
Choose $p,q\in \Theta_{hb}$ such that $p\neq \pm q$ and $p,q\neq\pm r,\pm s$. By applying (13) with $t=\pm p$ and $u=\pm q$ resp. $t=\pm q$ and $u=\pm p$ one obtains 
\begin{equation}
T_{\pm p,\pm q}(\sigma_{rs}\bar\sigma_{rr}a_1a_2a_3a_4),
T_{\pm q,\pm p}(\sigma_{rs}\bar\sigma_{rr}a_1a_2a_3a_4)\in H.
\end{equation}
It is an easy exercise to deduce (ii) from (14) (see the proof of Lemma \ref{lemsub1}).
\end{proof}

Lemmas \ref{4lemsub3} and \ref{4lemsub4} below are analogues of Lemmas \ref{lemsub3} and \ref{lemsub4}, respectively.

\begin{lemma}\label{4lemsub3}
Let $(I,\Omega)$ be an odd form ideal and $H$ a subgroup of $\U_{2n+1}(R,\Delta)$ normalised by $\EU_{2n+1}(I,\Omega)$. Then $T_{ij}(\sigma_{rs}a)\in H$ for any $i,j\in \Theta_{\hb}, i\neq \pm j$, $\sigma\in H$, $r,s
\in\Theta_{\hb},r\neq \pm s$ and $a\in I^7$.
\end{lemma}
\begin{proof}
Choose a $t\in \Theta_{hb}$ such that $t\neq\pm r,\pm s$. Let $a_0\in I$ and set 
\[\tau:=[\sigma^{-1},T_{tr}(-\bar\sigma_{rs}\bar a_0)]\in H.\]
Let $J$ be the involution invariant ideal generated by the set $\{a_0\sigma_{rs}\bar\sigma_{rr},a_0\sigma_{rs}\bar\sigma_{r,\pm t}\}$.
Clearly
\begin{align*}
\tau_{tt}&=1-\sigma'_{tt}\bar\sigma_{rs}\bar a_0\sigma_{rt}+\sigma_{t,-r}'\lambda^{(\epsilon(r)-\epsilon(t))/2}\sigma_{rs}a_0\sigma_{-t,t}\\
&=1-\sigma'_{tt}\underbrace{\bar\sigma_{rs}\bar a_0\sigma_{rt}}_{\in J}+\lambda^{-\epsilon(t)}\underbrace{\bar\sigma_{r,-t}\sigma_{rs} a_0}_{\in J}\sigma_{-t,t}\\
&\equiv 1\bmod J
\end{align*}
and 
\begin{align*}
\tau_{tr}&=-\sigma'_{tt}\bar\sigma_{rs}\bar a_0\sigma_{rr}+\sigma_{t,-r}'\lambda^{(\epsilon(r)-\epsilon(t))/2}\sigma_{rs}a_0\sigma_{-t,r}+\tau_{tt}\bar\sigma_{rs}\bar a_0\\
&=\sigma'_{tt}\underbrace{\bar\sigma_{rs}\bar a_0\sigma_{rr}}_{\in J}+\lambda^{-\epsilon(t)}\underbrace{\bar\sigma_{r,-t}\sigma_{rs} a_0}_{\in J}\sigma_{-t,r}+\tau_{tt}\bar\sigma_{rs}\bar a_0\\
&\equiv \bar\sigma_{rs}\bar a_0\bmod J
\end{align*}
by Lemma \ref{36}. Hence $\tau_{tt}\bar\tau_{tr}\equiv \sigma_{rs}a_0\bmod J$. By Lemma \ref{4lemsub2}(ii) we have $T_{ij}(\tau_{tr}\bar\tau_{tt}a)\in H$ for any $i,j\in \Theta_{\hb}, i\neq \pm j, a\in I^6$ whence $T_{ij}(\tau_{tt}\bar\tau_{tr}a)\in H$ for any $i,j\in \Theta_{\hb}, i\neq \pm j, a\in I^6$ (because of Relation (S1) in Lemma \ref{39}). Since $\tau_{tt}\bar\tau_{tr}a= \sigma_{rs}a_0a+xa$ for some $x\in J$, we obtain the assertion of the lemma in view of Lemma \ref{4lemsub2}.
\end{proof}

\begin{lemma}\label{4lemsub4}
Let $(I,\Omega)$ be an odd form ideal and $H$ a subgroup of $\U_{2n+1}(R,\Delta)$ normalised by $\EU_{2n+1}(I,\Omega)$. Let $(J,\Sigma)=U(H)$. Then 
$\EU_{2n+1}(J I^{10},\Omega_{\min}^{J I^{10}})\leq H$.
\end{lemma}
\begin{proof}
Let $\sigma\in \U_{2n+1}(R,\Delta)$ and $i,j,r,s\in\Theta_{\hb}$ such that $i\neq\pm j$ and $r\neq \pm s$. Furthermore, let $x,y\in J(\Delta)$. The previous lemma implies that 
\begin{enumerate}[(i)]
\item $T_{ij}(\sigma_{rs}a)\in H$ for any $a\in I^{7}$,
\item $T_{ij}(\sigma_{r,-r}a)\in H$ for any $a\in I^{8}$,
\item $T_{ij}(\sigma_{r0}xa)\in H$ for any $a\in I^{8}$,
\item $T_{ij}(\bar x \mu\sigma_{0s}a)\in H$ for any $a\in I^{8}$,
\item $T_{ij}((\sigma_{rr}-\sigma_{ss})a)\in H$ for any $a\in I^{8}$,
\item $T_{ij}((\sigma_{rr}-\sigma_{-s,-s})a)\in H$ for any $a\in I^{8}$ and
\item $T_{ij}(\bar x\mu (\sigma_{00}-\sigma_{ss})ya)\in H$ for any $a\in I^{9}$,
\end{enumerate}
cf. the proof of Lemma \ref{lemsub4}. It follows that $\EU_{2n+1}(J I^{10},\Omega_{\min}^{J I^{10}})\leq H$ in view of Relations (S5) and (SE2) in Lemma \ref{39}.
\end{proof}

\begin{theorem}\label{thmm4}
Let $(I,\Omega)$ be an odd form ideal and $H$ a subgroup of $\U_{2n+1}(R,\Delta)$ normalised by $\EU_{2n+1}(I,\Omega)$. Then 
\[\EU_{2n+1}(U(H)\ast I^{10})\leq H\leq \CU_{2n+1}((R,\Delta),U(H)).\]
\end{theorem}
\begin{proof}
The inclusion $H\subseteq \CU_{2n+1}((R,\Delta),U(H))$ holds by Proposition \ref{propul}. It remains to show the inclusion $\EU_{2n+1}(U(H)\ast I^{10})\subseteq H$. Because of Lemma \ref{4lemsub4} it suffices to show (i) and (ii) below.
\begin{enumerate}[(i)]
\item $T_{i}((Q(\sigma_{*s})\circ a)^{-\epsilon(i)})\in H$ for any $\sigma\in H$, $i,s\in\Theta_{\hb}$ and $a\in I^{10}$.
\item $T_{i}((((Q(\sigma_{*0})\minus (1,0))\circ y \+ (y,z)\minus (y,z)\circ \sigma_{ss})\circ a)^{-\epsilon(i)})\in H$ for any $\sigma\in H$, $i,s\in\Theta_{\hb}$, $(y,z)\in \Delta$ and $a\in I^{10}$.
\end{enumerate} 
The proof of (i) and (ii) above is essentially the same as the proof of (i) and (ii) in  the proof of Theorem \ref{thmm1} (of course one uses Lemmas \ref{4lemsub3} and \ref{4lemsub4} instead of Lemmas \ref{lemsub3} and \ref{lemsub4}, respectively).
\end{proof}

\begin{theorem}\label{thmm5}
Let $(I,\Omega)$ be an odd form ideal and $H$ a subgroup of $\U_{2n+1}(R,\Delta)$ normalised by $\EU_{2n+1}((R,\Delta),(I,\Omega))$. Then 
\[\EU_{2n+1}((R,\Delta),U(H)\ast I^{10})\leq H\leq \CU_{2n+1}((R,\Delta),U(H)).\]
\end{theorem}
\begin{proof}
See the proof of Theorem \ref{thmm2}.
\end{proof}

\begin{theorem}\label{thmm6}
Suppose that $H\lhd_d G$ where $d$ is a positive integer and $G$ a subgroup of $\U_{2n+1}(R,\Delta)$ containing $\EU_{2n+1}(R,\Delta)$. Suppose $U(H)=(I,\Omega)$. Then \[\EU_{2n+1}((R,\Delta),U(H)\ast I^{k})\leq H \leq \CU_{2n+1}((R,\Delta),U(H))\]
where $k=\frac{10^d-1}{9}-1$.
\end{theorem}
\begin{proof}
See the proof of Theorem \ref{thmm3}.
\end{proof}



\begin{remark}\label{remrefo}
Let $H$, $d$, $G$, $I$ and $k$ be as in Theorem \ref{thmm6}. Then Theorem \ref{thmm6} asserts that
\begin{equation}
\EU_{2n+1}((R,\Delta),U(H)\ast I^{k})\leq H.
\end{equation}
Clearly (15) is equivalent to 
\begin{equation}
U(H)\ast I^{k}\subseteq L(H)
\end{equation}
Note that also $L(H)\subseteq U(H)$ by Corollary \ref{corul3} and hence one gets a ``sandwich relation''. Moreover, one checks easily that (16) is equivalent to 
\begin{equation}
U(H)\subseteq L(H):I^k.
\end{equation}
Analogously one can reformulate the statements of Theorems \ref{thmm2}, \ref{thmm3} and \ref{thmm5}. For an analogous reformulation of the statements of Theorems \ref{thmm1} and \ref{thmm4} one should change the definition of $L(H)$.
\end{remark}

\end{document}